\documentclass[]{article}

\newcommand{\Op}{\mathcal{O}p}
\newcommand{\Pb}{\operatorname{Pb}}
\newcommand{\pb}{\operatorname{pb}}
\newcommand{\bp}{\operatorname{bp}}

\newcommand{\DDisc}{\mathbb{D}}
\newcommand{\eps}{\varepsilon}
\newcommand{\delt}{\delta}

\newcommand{\ellp}{\ell^\prime}
\newcommand{\ellpp}{\ell^{\prime\prime}}

\usepackage[utf8x]{inputenc}
\usepackage{amsmath}
\usepackage{amssymb}
\usepackage{amsfonts}
\usepackage{amsthm}
\usepackage{bm}
\usepackage{graphicx}
\usepackage{mathtools}
\usepackage{tikz-cd}
\usepackage{nicefrac}

\newtheorem{thm}{Theorem}[section]
\newtheorem{defn}[thm]{Definition}
\newtheorem{prop}[thm]{Proposition}
\newtheorem{clm}[thm]{Claim}
\newtheorem{cor}[thm]{Corollary}
\newtheorem{lem}[thm]{Lemma}
\newtheorem*{lem*}{Lemma}
\newtheorem*{thm*}{Theorem}
\newtheorem{thmdiff}{Theorem}
\theoremstyle{remark}
\newtheorem{rem}[thm]{Remark}

\makeatletter
\newcommand{\doublewidetilde}[1]{{%
		\mathpalette\double@widetilde{#1}%
	}}
	\newcommand{\double@widetilde}[2]{%
		\sbox\z@{$\m@th#1\widetilde{#2}$}%
		\ht\z@=.9\ht\z@
		\widetilde{\box\z@}%
	}
	\makeatother
	
\makeatletter
\newcommand{\triplewidetilde}[1]{{%
		\mathpalette\triple@widetilde{#1}
	}}
	\newcommand{\triple@widetilde}[2]{%
		\sbox\z@{$\m@th#1\doublewidetilde{#2}$}%
		\ht\z@=.9\ht\z@
		\widetilde{\box\z@}%
	}
	\makeatother

\newcommand{\VI}{V_{p^{}_1}^{X_{N-1}}}
\newcommand{\UIp}{{U^\prime}_{p^{}_1}^{X_{N-1}}}
\newcommand{\UIpp}{{U^{\prime\prime}}_{p^{}_1}^{X_{N-1}}}
\newcommand{\VNminI}{V_{p^{}_{N-1}}^{X_N}}
\newcommand{\UNminIp}{{U^\prime}_{p^{}_{N-1}}^{X_{N}}}
\newcommand{\UNminIpp}{{U^{\prime\prime}}_{p^{}_{N-1}}^{X_{N}}}

\DeclareRobustCommand{\DDelta}{{
	\mathchoice
	{\Delta\!\!\!\!{\scalebox{0.7}{\reflectbox{$\Delta$}}}}
    {\Delta\!\!\!\!{\scalebox{0.7}{\reflectbox{$\Delta$}}}}
	{\Delta\!\!\!\!\hspace{0.2pt}{\scalebox{0.55}{\reflectbox{$\Delta$}}}}
	{\Delta\!\!\!\hspace{-0.3pt}{\scalebox{0.4}{\reflectbox{$\Delta$}}}}
}
}

\usepackage{color}   
\usepackage{hyperref}
\hypersetup{
	colorlinks=true, 
	linktoc=all,     
	linkcolor=red,  
	linktocpage,		 
    citecolor=blue
}

\begin{document}
\title{ A homotopical viewpoint at the Poisson bracket invariants for tuples of sets.}
\author{Yaniv Ganor$^{1}$}
\footnotetext[1]{Partially supported by European Research Council Advanced grant 338809.}
\maketitle
\begin{abstract}
	We suggest a homotopical description of the Poisson bracket invariants for tuples of closed sets in symplectic manifolds.
	It implies that these invariants depend only on the union of the sets along with topological data.
\end{abstract}

\tableofcontents

\section{Introduction}
In \cite{BEP} Buhovski, Entov and Polterovich defined invariants of triples and quadruples of compact sets in a symplectic manifold, using certain variational problems involving the functional $(F,G) \mapsto \left\Vert\left\lbrace F, G\right\rbrace\right\Vert$. Specifically, for three compact sets $X,Y,Z$ in a symplectic manifold, $(M,\omega)$, the following invariant was defined:
\[ \pb_3(X,Y,Z) := \inf\left\lbrace \left\Vert \left\lbrace  F,G \right\rbrace\right\Vert \,\middle|\, F,G\in C^\infty_c(M), F\vert^{}_X \le 0, G\vert^{}_Y \le 0, (F+G)\vert^{}_Z \ge 1 \right\rbrace,\]
where $\{\cdot,\cdot\}$ is the Poisson bracket, and $\Vert \cdot \Vert$ is the supremum norm. Analogously, for four compact sets $X_0,X_1,Y_0,Y_1$ such that $X_0 \cap X_1 = Y_0 \cap Y_1 = \emptyset$ the following invariant was defined:
\[ \pb_4(X_0,X_1,Y_0,Y_1) := \inf\left\lbrace \left\Vert \left\lbrace  F,G \right\rbrace\right\Vert \,\middle|\, F,G\in C^\infty_c(M), 
\begin{array}{c}
	F\vert^{}_{X_0} \le 0  \\
	F\vert^{}_{X_1} \ge 1
\end{array} ,
\begin{array}{c}
	G\vert^{}_{Y_0} \le 0 \\
	G\vert^{}_{Y_1} \ge 1
\end{array}
 \right\rbrace.\]
In \cite{BEP} lower bounds for these invariants are computed. Such lower bounds tend to involve various flavors of holomorphic curves theory. In the same paper there are also proofs that the inequalities appearing in the definition can be replaced by equalities in some neighborhoods of the sets, which depend on the pair $(F,G)$, and that the functions can be assumed to be bounded between $0$ and $1$. Moreover, the definition of similar invariants for $n$-tuples of sets is sketched. These invariants are denoted by $\pb_n$.

The $\pb_4$ invariant is shown in \cite{BEP} to also have a dynamical interpertation in terms of time length of Hamiltonian chords, and in \cite{EGM} it was applied in the study of a topological invariant of Lagrangians.

In this paper we prove few results unifying $\pb_3$, $\pb_4$ and the general $\pb_n$. Moreover we show that $\pb_4$ depends only on the union of the sets, $X_0, X_1, Y_0$ and $Y_1$ which we denote by $X$, and on a first integer-valued cohomology class of $X$, which encodes the homotopical manner in which $X$ is decomposed into four sets. The first of our results is the following theorem:

\begin{thmdiff}\label{thm:pb3=pb4}
	Let $M$ be a symplectic manifold, and let $X_0,X_1,Y_0,Y_1$ be four compact subsets such that $X_0 \cap X_1 = Y_0 \cap Y_1 = \emptyset$.
	Then 
	\[
	\pb_4\left(X_0,X_1,Y_0,Y_1\right)=2\cdot\pb_3\left(X_0,Y_0,X_1\cup Y_1\right).
	\]
\end{thmdiff}
\begin{rem}
	In \cite{BEP} the following weaker inequality was proved:
	\[\pb_4\left(X_0,X_1,Y_0,Y_1\right)\ge \pb_3\left(X_0,Y_0,X_1\cup Y_1\right).\]
\end{rem}
Theorem \ref{thm:pb3=pb4} is in fact part of a more general phenomenon, namely, all the $\pb_n$ invariants introduced in \cite{BEP} can be reduced to $\pb_3$.
\begin{defn} 
	We say that $N$ sets, $X_1,\ldots,X_N$, intersect cyclically if ${X_i\cap X_j = \emptyset}$ whenever $i\not\in\{{j-1},j,{j+1}\}$ where $j-1$ and $j+1$ are computed cyclically $\operatorname{mod} N$.
\end{defn}
We define invariants, $\Pb_N$, of $N$-tuples of cyclically intersecting compact subsets of a symplectic manifold $M$. For compact $M$ they can be defined as follows:
Fix $\Delta$, a closed compact convex subset of $\mathbb{R}^2$ of $\operatorname{Area}(\Delta)=1$ with $\partial\Delta$ either smooth or polygonal. Fix $N$ points, $p_i\in\partial\Delta$, ordered cyclically counterclockwise. Denote by $\gamma_i$ the arc along $\partial\Delta$ emanating from $p_i$ towards the next point in the counterclockwise order. Define:
\[
\Pb_N(X_1,\ldots,X_N) := \inf\left\{\left\Vert\left\lbrace\Phi_1,\Phi_2\right\rbrace\right\Vert \,\middle|\,
\begin{tabular}{@{}l@{}}
$\Phi=(\Phi_1,\Phi_2)\colon M\to\Delta\subset\mathbb{R}^2$, such that\\ $\forall i, \exists \underset{\text{open}}{U_i}\supset X_i$ such that $\Phi\left(U_i\right)\subseteq \gamma_i$
\end{tabular}
\right\}.
\]
We show that resulting quantity, $\Pb_N(X_1,\ldots,X_N)$, depends neither on the choice of the domain $\Delta$ nor on the choice of the points $p_i\in\partial\Delta$ as long as they are chosen with the same cyclical order. Moreover, we show that $2\pb_3(X,Y,Z)= \Pb_3(X,Y,Z)$ and that $\pb_4(X_0,X_1,Y_0,Y_1)=\Pb_4(X_0,Y_0,X_1,Y_1)$. (Note that for $\Pb_N$ we list the sets in a cyclical order). In fact up to multiplication by a constant, which is due to the normalization of the area of $\Delta$, our $\Pb_n$ is equal to $\pb_n$ from \cite{BEP}. We prove in \textbf{Theorem \ref{thm:actualThm1}} the following statement:\\ 
\textit{For $N\ge 4$, and $X_1,\ldots,X_N$ intersecting cyclically, it holds that
\[
	\Pb_N\left(X_1,\ldots,X_N\right)=\Pb_{N-1}\left(X_1,\ldots,X_{N-1}\cup X_N\right),
\]}
from which Theorem \ref{thm:pb3=pb4} follows.
This reduces from $\Pb_{N}$ to $\Pb_{N-1}$. As a partial converse, we show that one can also go the other way, namely, recover $\Pb_{N}$ from the data of $\Pb_{N+1}$s, replacing the intersection $X_1 \cap X_N$ with a compact neighborhood which is added as a new set to the tuple, and taking limit over such neighborhoods.
\begin{thmdiff}\label{thm:pb3 limit}
	Let $X_1,\ldots,X_N$ be compact sets intersecting cyclically and if ${N=3}$ assume also that $X_1\cap X_2 \cap X_3=\emptyset$. Let $K_n$ be a decreasing sequence of compact neighborhoods of $X_1 \cap X_N$, converging to $X_1\cap X_N$ in the Hausdorff distance. Moreover assume that $K_1\cap \left(\bigcup_{j=2}^{N-1} X_j\right)=\emptyset$. Then the following limit exists and equals $\Pb_N(X_1,\ldots,X_N)$:
	\[
	\lim_{K_n\searrow X_1\cap X_N} \Pb_{N+1}(\overline{X_1 \setminus K_n},X_2,\ldots,X_{N-1},\overline{X_N\setminus K_n}, K_n) = \Pb_N(X_1,\ldots,X_N).
	\]
\end{thmdiff}
Theorem \ref{thm:pb3 limit} allows us to relate $\pb_3$ to dynamics, see Corollary \ref{cor:pb3_dynamics}.

Together these theorems unify $\pb_3$, $\pb_4$ and their generalizations, $\pb_n$, which were defined in \cite{BEP}.
As further unification we prove that $\Pb_n$ depends only on the union of the sets $X_i$ and on some homotopical data, namely, we show that $\Pb_n$ defines a function on the set of homotopy classes of maps from the compact set $X=X_1\cup\ldots\cup X_N$ to $S^1$ (which by \cite{Morita1975} equals $H^{1}(X;\mathbb{Z})$, the first cohomology of the constant sheaf $\mathbb{Z}$), where the homotopy class (first cohomology class) describes the manner in which $X$ is decomposed.
We will use the two viewpoints on $H^{1}(X;\mathbb{Z})$, as either first integral-cohomology or homotopy classes of maps from $X \to S^1$ interchangeably.
\begin{defn}
	Denote by $\overline{B_1}$ the unit ball in $\mathbb{R}^2$ and denote by $S^1$ its boundary. Let $M$ be a symplectic manifold and $X$ be a compact subset. For any ${\alpha\in H^1(X;\mathbb{Z})}$ define:
	\[Pb^{}_X(\alpha):=\inf\left\lbrace 
	\left\Vert \left\lbrace \phi_1, \phi_2 \right\rbrace \right\Vert
	\,\middle|\,
	\begin{tabular}{@{}l@{}}
		 $\phi=(\phi_1,\phi_2)\colon M\to \overline{B_1}$ such that \\ $\phi_1,\phi_2$ have compact support, and \\
		 $\exists \underset{\text{open}}{U}\supset X,\ \phi(U)\subseteq S^1,\  \left[ \phi\vert^{}_X\right]=\alpha$
	\end{tabular}
	\right\rbrace,\]
	where $\left[ \phi\vert^{}_X\right]$ denotes the homotopy class of the function $\phi\vert^{}_X$.
\end{defn}

Let $X$ be be a compact subset of a manifold $M$, and assume $X=X_1\cup X_2 \cup X_3$ where $X_1\cap X_2 \cap X_3 = \emptyset$ and each $X_k$ is compact. Denote by $\alpha\in H^1(X;\mathbb{Z}))$ the class determined by the decomposition $X=X_1\cup X_2 \cup X_3$, in the sense that $\alpha = [f]$ where $f$ is a function, $f\colon X \to S^1$, such that $f\vert^{}_{X_i}\subset \gamma_i$ where $S^1 = \gamma_1 \cup \gamma_2 \cup \gamma_3$ is a decomposition into three consecutive arcs ordered counterclockwise. The class $\alpha$ depends neither on the decomposition of the circle into arcs (up to a cyclical order preserving relabeling) nor on the particular function $f$ chosen. We prove the following:
\begin{thmdiff}\label{thm:htpyCharacterization}
	\[Pb_3(X_1,X_2,X_3)=Pb^{}_X(\alpha).\]
\end{thmdiff}
\begin{rem}
	A similar proof would work for any $\Pb_n$, thereby providing another proof for Theorem \ref{thm:pb3=pb4}, by equating all $\Pb_n$ with the same $\Pb^{}_X(\alpha)$.
	We chose to separate the proof of Theorem \ref{thm:pb3=pb4} and give a direct proof of it first, since it is a simpler proof, and moreover, restricting to $n=3$ somewhat simplifies the discussion on decompositions of $X$ versus homotopy classes of maps from $X$ to $S^1$
\end{rem}
As an application for Theorem \ref{thm:htpyCharacterization} we prove subhomogeneity of $\Pb_X(\alpha)$, a fact which has repercussions for $\pb_4$:
\begin{thmdiff}\label{thm:pbaSubhomogeneous}
	Let $X$ be compact subset of a symplectic manifold $M$. Then for all $0\ne \alpha \in H^1(X;\mathbb Z)$ and for all $0<k\in \mathbb N$ we have:
	\[
	\Pb^{}_X(k\alpha) \le k \cdot \Pb^{}_X(\alpha).
	\]
\end{thmdiff}
This result is motivated by the work of \cite{EGM} on Lagrangian topology. 
In \cite{EGM} an invariant associated to Lagrangian submanifolds admitting fibrations over $S^1$ was introduced, named $\bp_L$, whose definition is based on $\pb^{+}_4$ (a refinement of $\pb_4$ see Remark \ref{rem:pb4plus}). For a Lagrangian $L$ admitting a smooth fibration, $f\colon L\to S^1$, one cuts $S^1$ into four consecutive arcs, denotes their preimages under $f$ by $X_0,Y_0,X_1,Y_1$ and computes 
\[\bp_L(f) := \frac{1}{pb^{+}_4(X_0,X_1,Y_0,Y_1)}.\]
The quantity $\bp_L(f)$ depends neither on the isotopy class of the smooth fibration, $f$, nor on the particular choice of arcs in $S^1$ (keeping the same cyclical ordering). For Lagrangian tori of dimensions $2$ and $3$ these classes of fibrations correspond to first integral cohomology classes of $L$, thus $\bp_L$ defines a function, $H^{1}(L;\mathbb Z)\colon L \to \mathbb [0,\infty)$.
The invariant, $\bp_L$, is shown to be smaller or equal to another invariant of $L$, named $\operatorname{def}_L$, defined in terms of Lagrangian isotopies with prescribed flux, which gives function on the real-valued first cohomology,  $\operatorname{def}_L\colon H^{1}(L;\mathbb{R}) \to \mathbb{R}$. The function $\operatorname{def}_L$ is seen to be $\mathbb{R}_+$-homogeneous immediately from the definition and in all examples in \cite{EGM} where it was manageable to compute both invariants they turned out to be equal. This raised the question whether always $\operatorname{def}_L\vert^{}_{ H^{1}(L;\mathbb{Z})} = \bp_L$.
Therefore it is interesting to study homogeneity of $\bp_L$, as it may provide evidence in deciding the question.
Moreover, our definition of $\Pb^{}_X(\alpha)$ allows one to extend the definition of $\bp_L$ as a function on the first integral cohomology for general Lagrangians without any regard to fibrations or issues of smoothness of isotopies (which were the reason why \cite{EGM} had to restrict to tori of dimension $2$ and $3$).

\begin{rem}\label{rem:pb4plus}
	In \cite{EP} Entov and Polterovich defined a refinement of $\pb_4$ replacing the supremum norm with maximum.
	\[ \pb^{+}_4(X_0,X_1,Y_0,Y_1) := \inf\left\lbrace \max \left\lbrace  F,G \right\rbrace \,\middle|\, F,G\in C^\infty_c(M), 
	\begin{array}{c}
	F\vert^{}_{X_0} \le 0  \\
	F\vert^{}_{X_1} \ge 1
	\end{array} ,
	\begin{array}{c}
	G\vert^{}_{Y_0} \le 0 \\ 
	G\vert^{}_{Y_1} \ge 1
	\end{array}
	\right\rbrace.\]
	It was used in \cite{EP} to study dynamics, as the invariant detects both the existence and the direction of Hamiltonian chords.
	Since $\pb_4^{+}$ is always nonnegative (Compact support guarantees a point with vanishing derivative of either $F$ or $G$, therefore at that point $\left\lbrace F,G \right\rbrace=0$, hence $\displaystyle\max_M\left\lbrace F,G \right\rbrace\ge0$), and since all our proofs involve upper bound inequalities with respect to nonnegative quantities, one could replace $\Vert \cdot \Vert$ by $\max(\cdot)$ and obtain analogous theorems for $\pb^+$.
\end{rem}

\textbf{Acknowledgments:} I would like to thank my advisor, Leonid Polterovich for suggesting to look at the theory of Poisson bracket invariants, $\pb_n$ for $n\ge5$.
I also thank him together with Michael Entov and Lev Buhovoski for reading the preprints of this paper and providing important feedback and comments.
%

\section{$\eps$-pseudoretracts}
One of our main tools for manipulating functions without increasing the Poisson bracket too much, is by post-composing with a function from $\mathbb{R}^2$ to $\mathbb{R}^2$ with a bound on the Jacobian, therefore we introduce the following notion:
\begin{defn}
	Let $\Delta\subset \mathbb{R}^2$ be a compact set in the plane.
	We call a smooth map $T=(T_1,T_2)\colon \mathbb{R}^2 \to \Delta$ an $\eps$-pseudoretract onto $\Delta\subset \mathbb{R}^2$ if
	\begin{itemize}
		\item $T$ is onto $\Delta$.
		\item $T$ maps $\mathbb {R}^2\setminus\Delta$ to $\partial \Delta$.
		\item $\vert DT \vert\le 1+\eps$ (Where $\vert DT\vert $ is the Jacobian determinant of $T$, \\ namely, $\frac {T^*\omega_{\mathbb {R}^2}}{\omega_{\mathbb {R}^2}}$).
	\end{itemize}
	In particular it holds for an $\eps$-pseudoretract that $\Vert\{T_1,T_2\}\Vert\le 1+\eps$.
\end{defn}
\begin{prop}\label{prop:pseudoretract_properties}
	Let $\Phi=(\Phi_1,\Phi_2)\colon M \to \mathbb{R}^2$ and let ${T=(T_1,T_2)\colon \mathbb{R}^2 \to \Delta}$ be an $\eps$-pseudoretract onto $\Delta\subseteq \mathbb{R}^2$. Consider ${T\circ\Phi = \left(\left(T\circ\Phi\right)_1,\left(T\circ\Phi\right)_2 \right)\colon M \to \mathbb{R}^2}$
	Then:
	\begin{enumerate}
		\item $\left\Vert\left\{\left(T\circ\Phi\right)_1,\left(T\circ\Phi\right)_2\right\}\right\Vert \le (1+\eps)\left\Vert\left\{\Phi_1,\Phi_2\right\}\right\Vert\ $.
		\item For all $x\in M$ such that $\Phi(x)\in \mathbb{R}^2 \setminus \Delta$:
		\[ 
		\left\{\left(T\circ\Phi\right)_1,\left(T\circ\Phi\right)_2\right\} = 0.
		\]
	\end{enumerate}
\end{prop}
\begin{proof}
	Statement (1) follows from the fact that $\vert DT \vert\le 1+\eps$, and the formula in the proof of Claim \ref{clm:symp_invariance}. Statement (2) follows from the fact that $\partial \Delta$ is one dimensional, so locally around $x$, the function $\left(T\circ\Phi\right)_1$ is a function of $\left(T\circ\Phi\right)_2$ or vice-versa, hence the Poisson bracket vanishes.
\end{proof}

The following corollary will be useful when dealing with $\Pb^{}_X(\alpha)$:
\begin{cor}\label{cor:pb_bound_retract}
	Let $X$ be a compact subset of a symplectic manifold $M$. Let $\Phi\colon M \to \overline{B_1}$ be a function such that $\Phi(X)\subset \partial(\overline{B_1})=S^1$. Let $\eps>0$  and $K>0$ be positive numbers such that on $\Phi^{-1}\left(B_{1-\eps}\right)$ we have a bound $\left\Vert \left\{ \Phi_1,\Phi_2 \right\} \right\Vert \le K$. Then there exists $\Psi\colon M \to \overline{B_1}$, such that:
	\begin{itemize}
		\item $\Psi\vert^{}_X \equiv \Phi\vert^{}_X$.
		\item $\left\Vert\left\{ \Psi_1, \Psi_2 \right\}\right\Vert \le \frac {1+\eps}{1-\eps}K$.
	\end{itemize}
\end{cor}
\begin{proof}
	The proof is an immediate application of Proposition $\ref{prop:pseudoretract_properties}$, pick an $\eps$-pseudoretract, $T$, onto $B_{1-\eps}$, and consider $\widetilde{\Phi}:=\mathcal{H}_{\frac{1}{1-\eps}} \circ T \circ \Phi$, where $\mathcal{H}_{\frac{1}{1-\eps}}$ is the homothety by a factor of $\frac{1}{\sqrt{1-\eps}}$.
\end{proof}

\begin{prop}\label{prop:pseudoretractsExist}
	$\eps$-pseudorectracts onto $\Delta$ exist for any compact convex $\Delta$ with either a smooth boundary or a polygonal boundary.
\end{prop}
\begin{proof}
	For a set $\Delta$ such that $\partial\Delta$ is smooth, WLOG assume $(0,0)\in \operatorname{Int} \Delta$ and parametrize $\partial\Delta$ in polar coordinates by $r^{}_{\partial\Delta}(\theta)e^{i\theta}$ where  $r^{}_{\partial\Delta} \colon [0,2\pi] \to (0,\infty)$ is the radial coordinate of the boundary.
	Let $\rho\colon [0,\infty) \to [0,\infty)$ be a function such that $\rho(x)\vert^{}_{[0,1/2]} = x$, $\rho\vert^{}_{[1,\infty)} = 1$ and $0\le\rho^\prime \le 1+\eps$.
	Now define $T\colon \mathbb{R}^2 \to \mathbb{R}^2$ in polar coordinates by:
	\[ 
	T(re^{i\theta}) = r^{}_{\partial\Delta}\rho\left(\frac{r}{r^{}_{\partial\Delta}(\theta)}\right)e^{i\theta}.
	\]
	The proof continues similarly to the proof of Lemma 2.4 in \cite{BEP}, we cite the formula for the Jacobian appearing in \cite{BEP}:
	\[\frac {T^*\omega_
		{\mathbb {R}^2}}{\omega_{\mathbb {R}^2}} = \frac {\vert T(re^{i\theta})\vert }{r}\frac{\partial }{\partial r}\vert T(re^{i\theta}) \vert.\]
	Where $\vert\cdot\vert$ is the distance to the origin. It follows that $T$ is the desired pseudoretract if $\eps$ is small enough.
	For sets, $\Delta$, whose boundary is a triangle an argument appears in \cite{BEP}. A similar argument works for every polygon and we describe it briefly.
	Let $\Delta$ be a convex polygon with $N$ edges denoted by $\ell_1,\ldots,\ell_N$. For every $1\le k \le N$ we construct a function $T_k$ as follows:
	Denote by $\ellp$ and $\ellpp$ the edges adjacent to $\ell_k$, such that they are oriented counterclockwise as $\ellp,\ell_k,\ellpp$,
	Denote by $\widetilde{\ell_k}$ the line in the plane on which $\ell_k$ lies, and denote by $H_k$ the half plane containing $\Delta$ whose boundary is $\widetilde{\ell_k}$.
	We now define $T_k$ by cases:
	\begin{description}
		\item [Case 1: $\ellp \parallel \ellpp$]
		\quad\\
		Let $u,v$ be unit vectors such that $u \parallel \ell_k$ and $v \parallel \ellpp$ where both vectors are oriented by the same orientations of $\ell_k$ and $\ellpp$ that are induced by the counterclockwise orientation of $\partial \Delta$.
		Let $(x,y)$ denote coordinates in $\mathbb{R}^2$ with respect to the basis $\left\{u,v\right\}$ and such that $(0,0) = \ell_k \cap \ellpp$.
		Let $\rho\colon (-\infty,\infty) \to [0,\infty)$ be a function such that $\rho(x)\vert^{}_{[\eps,\infty]} = x$, $\rho\vert^{}_{(-\infty,0]} = 0$ and $0\le\rho^\prime \le 1+\eps$.
		Define 
		\[T_k(x,y) := (\rho(x),y). \]
		Thus $T_k$ maps $H_k$ to $\partial H_k$ and $T_k(\Delta) = \Delta$ if $\eps$ is small enough.
		
		\item [Case 2: The continuations of $\ellp$ and $\ellpp$ intersect at a point $p \in H_k$.]
		\quad\\
		Pick polar coordinates $(r,\theta)$ such that $p$ is the center of the coordinate system. 
		Let $J_k = \left\{\theta \,\middle|\,\; \exists r \text{ such that } re^{i\theta}\in \widetilde{\ell_k}\right\}$ 
		and for every $\theta\in J_k$ denote by $R_k(\theta)$ the distance from $p$ to the intersection of
		$\widetilde{\ell_k}$ with the ray emanating from $p$ with angle $\theta$. Let $\rho\colon [0,\infty) \to [0,\infty)$ be a function such that $\rho(x)\vert^{}_{[0,1-\eps]} = x$, $\rho\vert^{}_{[1,\infty)} = 1$ and $0\le\rho^\prime \le 1+\eps$.
		Define
		\begin{align*}
		T_k(r,\theta):=
		\begin{cases*}
		R_k(\theta)\rho\left(\frac{r}{R_k(\theta)}\right)e^{i\theta} & 	$\theta \in J_k$ \\
		re^{i\theta} &   Otherwise.
		\end{cases*}		
		\end{align*} 
		Thus $T_k$ maps $H_k$ to $\partial H_k$ and $T_k(\Delta) = \Delta$ if $\eps$ is small enough.
		
		\item [Case 3: The continuations of $\ellp$ and $\ellpp$ intersect at a point $p \not\in H_k$.]
		\quad\\
		Again pick polar coordinates $(r,\theta)$ such that $p$ is at the center of the coordinate system. 
		Let $J_k = \left\{\theta \,\middle|\,\; \exists r \text{ such that } re^{i\theta}\in \widetilde{\ell_k}\right\}$ 
		and for every $\theta\in J_k$ denote by $R_k(\theta)$ the distance from $p$ to the intersection of
		$\widetilde{\ell_k}$ with the ray emanating from $p$ with angle $\theta$. Let $\rho\colon [0,\infty) \to [0,\infty)$ be a function such that $\rho(x)\vert^{}_{[1+\eps,\infty)} = x$, $\rho\vert^{}_{[0,1]} = 1$ and $0\le\rho^\prime \le 1+\eps$.
		Define:
		\[
		T_k(r,\theta):=
		R_k(\theta)\rho\left(\frac{r}{R_k(\theta)}\right)e^{i\theta}.
		\]
		In this case $T_k$ is defined only on the open half plane containing $\Delta$ whose boundary is the line parallel to $\widetilde{\ell_k}$ passing through $p$. Denote this half plane by $\widetilde{H_k}$. The map $T_k$ maps $\widetilde{H_k}$ to $\partial H_k$ and $T_k(\Delta) = \Delta$ if $\eps$ is small enough.
	\end{description}
	To get the desired pseudoretract one takes 
	\[
		T := T_N \circ \ldots T_1 \circ S.
	\]
	Where $S$ is a pseudorectact on a smooth convex body containing $\Delta$ such that all $T_k$ are defined on it.
	Choosing $\eps$ small enough yields the desired function.	
\end{proof}
\begin{rem}\label{rem:pseudoretract_property}
	The pseudoretract of $\mathbb{R}^2$ onto a polygon described above has the following property:
	Consider the angle opposite to the interior angle at a vertex $v$, that is, the angle formed at a vertex $v$ by continuation of the adjacent edges, and denote by $A_v$ the plane sector formed by it. Then there exists $\eps>0$ such that $B_\eps(v)\cap A_v$ is mapped to $\{v\}$ by the pseudoretract.
\end{rem}

\section{The invariants $\Pb_n$}	
\subsection{Definitions and Setup}
Let $(M,\omega)$ be a symplectic manifold and let $X_1,\ldots,X_N\subseteq M$ a collection of compact subsets intersecting cyclically. We let $\DDelta=(\Delta,p_1,\ldots,p_N)$ denote the following data:
\begin{enumerate}
	\item $\Delta$ is a closed compact convex subset of $\mathbb{R}^2$ of $\operatorname{Area}(\Delta)=1$ with $\partial\Delta$ either smooth or polygonal.
	\item $p_i\in\partial\Delta$ are marked points ordered cyclically counterclockwise.  
\end{enumerate}
We denote by $\gamma_i$ the arc along $\partial\Delta$ emanating from $p_i$ towards $p_{i+1}$ ($i+1$ is computed cyclically $\operatorname{mod} N$). To also incorporate non-compact symplectic manifolds, we define the following condition:
\begin{defn}
	We say that a function $\Phi\colon M \to \Delta$ satisfies the (CS)-condition if there exists $p=(p_1,p_2)\in \Delta$ such that $\Phi_1 - p_1\colon M \to \mathbb{R}$ and $\Phi_2 - p_2\colon M \to \mathbb{R}$ are both compactly supported. When $M$ is compact this is automatically satisfied for all $\Phi$.
\end{defn}
Put:
\[
	\mathcal{F}_{\DDelta,N}^\prime(X_1,\ldots,X_N) = \left\{\Phi\colon M\to\Delta \,\middle|\, \Phi \text{ is (CS) \& } \forall i,\; \exists \underset{\text{open}}{U_i}\supset X_i\; \text{such that}\; \Phi\left(U_i\right)\subseteq \gamma_i\right\}.
\]

\begin{defn}
	Define: \[
	\Pb_N^\DDelta(X_1,\ldots,X_N)=\underset{\Phi\in\mathcal{F}_{\DDelta,N}^\prime}{\inf}{\left\Vert\left\{\Phi_1,\Phi_2\right\}\right\Vert}.
	\]
	Where $\Phi_1,\Phi_2$ are the components of $\Phi\colon M\to \Delta \subset \mathbb{R}^2$, i.e. $\Phi$ is given by $\Phi(x)=\left(\Phi_1\left(x\right),\Phi_2\left(x\right)\right)$.
\end{defn}

Also denote: 
\[
	\mathcal{F}_{\DDelta,N}(X_1,\ldots,X_N) = \left\{\Phi\colon M\to\Delta \,\middle|\, \Phi \text{ is (CS) \& } \forall i,\; \Phi\left(X_i\right)\subseteq \gamma_i\right\}.
\]
\begin{rem}
	By a method of homotheties and what we call pseudoretracts, one has:
	\[
		\underset{\Phi\in\mathcal{F}_{\DDelta,N}^\prime}{\inf}{\left\Vert\left\{\Phi_1,\Phi_2\right\}\right\Vert} = \underset{\Phi\in\mathcal{F}_{\DDelta,N}}{\inf}{\left\Vert\left\{\Phi_1,\Phi_2\right\}\right\Vert}.
	\]
	See an analogous proof in \cite{BEP} for $\pb_3$, and Step 2 in the proof of Claim \ref{prop:pb_invariance}.
\end{rem}

\begin{rem}
	When some pieces of the data are clear from the context (for example the sets $(X_1,\ldots,X_N)$, the number of sets, etc') we  omit them from the notation of $\Pb$ or $\mathcal{F}$ in favor of a less cluttered notation.
\end{rem}

\begin{rem}
	In \cite{BEP} (Proposition 1.3) it is shown that the $\pb$-invariant can also be defined in terms of bounded functions:
	\[ \pb_3(X,Y,Z) = \inf\left\lbrace \left\Vert \left\lbrace  F,G \right\rbrace\right\Vert \,\middle|\,
	\begin{tabular}{@{}l@{}l@{}}
		 $F,G\in C^\infty_c(M)$, $F\ge0$, $G\ge 0$, $F+G\le 1$, \\
		 $\exists \underset{\text{open}}{U_X}\supset X$, $\underset{\text{open}}{U_Y}\supset Y$, $\underset{\text{open}}{U_Z}\supset Z$ such that\\
		 $F\vert^{}_{U_X}=0$, $G\vert^{}_{U_Y}=0$, $(F+G)\vert^{}_{U_Z}=1$
	\end{tabular}
	 \right\rbrace,\]
	\[ \pb_4(X_0,X_1,Y_0,Y_1) = \inf\left\lbrace \left\Vert \left\lbrace  F,G \right\rbrace\right\Vert \,\middle|\,
	\begin{tabular}{@{}l@{}l@{}}
	$F,G\in C^\infty_c(M)$, $0\le F\le 1$, $0\le G\le 1$ \\
	$\exists \underset{\text{open}}{U_{X_i}}\supset X_i$, $\underset{\text{open}}{U_{Y_i}}\supset X_i$ such that\\ 
	$F\vert^{}_{U_{X_0}}=0$, $F\vert^{}_{U_{X_1}}=1$, $G\vert^{}_{U_{Y_0}}=0$, $G\vert^{}_{U_{Y_1}}=1$
	\end{tabular}
	\right\rbrace.\] 
	As a consequence we get that by definition $\pb_4(X_0,X_1,Y_0,Y_1) = \Pb_4^\DDelta(X_0,Y_0,X_1,Y_1)$ for $\Delta$ being the square with side length $1$ and $p_i$ its vertices, and that $\pb_3(X,Y,Z)=\frac {1}{2}\Pb_3^\DDelta(X,Y,Z)$, where $\Delta$ is a right triangle with legs of length $\sqrt{2}$ and $p_i$ are its vertices. The emergence of the factor of $\nicefrac{1}{2}$ in the formula is due to us working with the normalization of mapping into domains of area $1$. Note that to deduce the above equalities for a non-compact $M$ one needs to use the methods of Proposition \ref{prop:pb_invariance} and of Theorem \ref{thm:pb3=pb4} to move the point $p\in \Delta$ witnessing the (CS) condition such that $p=p_1=(0,0)$ and only then the (CS) condition coincides with the requirement that the function $F$ and $G$ from $\pb_3$ or $\pb_4$ have a compact support.
\end{rem}
\begin{clm}\label{clm:symp_invariance}
	Let $S\colon \Delta_1 \to \Delta_2$ is a symplectomorphism where $\Delta_1,\Delta_2$ are domains in $\mathbb{R}^2$, and let $\Phi\colon M \to \Delta_1$ be a smooth map.\\ Denote by $\Psi:=S\circ\Phi\colon M \to \Delta_2$ their composition. Then $\left\{\Phi_1,\Phi_2\right\}=\left\{\Psi_1,\Psi_2\right\}$. 
\end{clm}
\begin{proof}
	This follows from the description of Poisson bracket $\left\{\Phi_1,\Phi_2\right\}$ as 
	\[ \left\{\Phi_1,\Phi_2\right\} =  -n\frac{d\Phi_1\wedge d\Phi_2 \wedge\omega^{\wedge n-1}}{\omega^{\wedge n}}.\]
	and from $d\Phi_1\wedge d\Phi_2 = \Phi^*\omega_{\mathbb{R}^2}$.
\end{proof}

\subsection{Independence of $\DDelta$}
We denote by $B_\delt(p)$ the open ball in $\mathbb{R}^2$ of radius $\delta$ centered at $p$. When $p$ is the origin we omit it and just write $B_\delta$. We denote by $\overline{B_\delt(p)}$ the closed ball.
\begin{prop}\label{prop:pb_invariance}
	For all $\DDelta_1,\DDelta_2$ we have $\Pb_N^{\DDelta_1}(X_1,\ldots,X_N)=Pb_N^{\DDelta_2}(X_1,\ldots,X_N)$,
	that is, $\Pb_N^\DDelta$ neither depends on $\Delta$ nor on the points $p_i\in\partial\Delta$ as long as their cyclical order is preserved.
\end{prop}
\begin{proof}\quad\\
		\textbf{Step 1:} 
		Let $\DDelta, \DDelta^\prime$ denote the same domain $\Delta$ with \textbf{smooth} boundary, and different sets of points, $\{p_i\}$ and $\{p^\prime_i\}$, both ordered cyclically counterclockwise.
		$\partial\Delta$ is a Lagrangian submanifold in $\mathbb {R}^2$ and by Weinstein neighborhood theorem $\partial\Delta$ has a neighborhood $U$ symplectomorphic to a neighborhood, $V$, of the zero-section in $T^*\partial\Delta$, such that $\partial \Delta$ is identified with the zero section. Any vector field, $X$, along the zero section can be extended to a compactly supported Hamiltonian vector field in $V$, by setting in the canonical coordinates $(q,p)$, $H(q,p):=p(X(q))$ and then multiplying with a suitable cutoff function.
		Thus by extending an appropriate vector field along $\partial\Delta$ to a Hamiltonian vector field in $\mathbb{R}^2$ one gets a symplectomorphism, $S$, preserving $\Delta$ and mapping each point $p_i$ to $p^\prime_i$.  Hence, by Claim \ref{clm:symp_invariance} we deduce the independence of $\Pb$ on the points $p_i$ when $\Delta$ has smooth boundary, and we can omit them from the notation from now on.		
\\\\ \textbf{Step 2:}
		We compare $\Pb$ defined with respect to $\DDelta_1$ and $\Pb$ defined with respect to a unit disc $\overline{B_1}$.
		Let $\Phi\in \mathcal{F}_{\DDelta_1,N}^\prime$. Consider the deflated disc $\overline{B_{1-\eps}}$, by the  Dacorogna-Moser theorem \cite{DM} (The theorem essentially states the existence of a volume-form preserving map between domains of equal total volume) there exists a symplectomorphism, $S$, mapping it to a domain in the interior of $\Delta_1$. Let $T$ denote some smooth $\eps$-pseudoretract onto $S\left(\overline{B_{1-\eps}}\right)$.
		Let $\mathcal H_{\frac{1}{1-\eps}}$ be the homothety by a factor of $\frac {1}{\sqrt{1-\eps}}$ (Note that we chose to denote homotheties by their effect on areas, not on length). Define:
		\[
			\Psi:= \mathcal H_{\frac{1}{1-\eps}} \circ S^{-1}\circ T \circ \Phi.
		\]
		Then $\Psi \in \mathcal{F}_{\DDisc}$, where $\DDisc = \left(\overline{B_1},\left\lbrace\Psi(p_i)\right\rbrace\right)$, i.e the unit disc with the resulting configuration of points, $\{\Psi(p_i)\}$. We have: 
		\[
		\left\Vert\left\{\Psi_1,\Psi_2\right\}\right\Vert \le \frac {1+\eps}{1-\eps} \left\Vert\left\{\Phi_1,\Phi_2\right\}\right\Vert.
		\] Sending $\eps \to 0$ yields ${Pb_N^{\DDelta_1} \ge \Pb_N^{\DDisc}}$, in light of Step 1 which takes care of the points along $\partial \Delta$.
		
		Conversely, let $\Phi\in \mathcal{F}_{\DDisc}^\prime$, for $\DDisc=\left(\overline{B_1},\{p_i\}\right)$, with $\{p_i\}$ being arbitrary points along the boundary, ordered counterclockwise. Let $\mathcal H_{1+\eps}$ be the homothety by a factor of $\sqrt{1+\eps}$. Let $S$ be a symplectomorphism mapping $\overline{B_{1+\eps}}$ to a domain whose interior contains $\Delta_1$, and let $T$ be an $\eps$-pseudoretract of $\mathbb{R}^2$ onto $\Delta_1$.
		Define:
		\[
			\Psi:= T \circ S \circ \mathcal H_{1+\eps} \circ \Phi.
		\]
		Then $\Psi \in \mathcal{F}_{\DDelta_1}$ for $\DDelta_1 = (\Delta_1, \{\Psi(p_i)\})$ and \[\left\Vert\left\{\Psi_1,\Psi_2\right\} \right\Vert \le (1+\eps)^2 \left\Vert\left\{\Phi_1,\Phi_2\right\}\right\Vert.\]
		Sending $\eps \to 0$ yields ${Pb_N^{\DDelta_1} \le \Pb_N^{\DDisc}}$.
\\\\ \textbf{Step 3:}
		If $\Phi$ satisfies (CS) then also $\Psi$ does, since $\Psi$ is obtained from $\Phi$ by postcomposition.
\\\\ \textbf{Step 4:}
		By moving the points along the boundary of a disc as in Step 1 and combining with the pseudoretracts of Step 2, we deduce the independence of $\Pb$ of the points $p_i$ also when $\partial\Delta$ is a polygon.
\end{proof}
\begin{rem}
	Since $\Pb^\DDelta_N$ is independent of $\DDelta$, when we won't care about the details of the implementation we will suppress $\DDelta$ from the notation and just write $\Pb_N$.
\end{rem}
\subsection{Proof of Theorem \ref{thm:pb3=pb4}}
We recall the statement of the theorem:
\begin{thm}\label{thm:actualThm1}
	For $N\ge 4$, and $X_1,\ldots,X_N$ compact sets intersecting cyclically, it holds that:
	\[
		\Pb_N\left(X_1,\ldots,X_N\right)=\Pb_{N-1}\left(X_1,\ldots,X_{N-1}\cup X_N\right).
	\]
\end{thm}
\begin{proof}
	The inequality $\Pb_N\left(X_1,\ldots,X_N\right)\ge\Pb_{N-1}\left(X_1,\ldots,X_{N-1}\cup X_N\right)$ follows from forgetting the $n$-th point, that is, any function $\Phi\in\mathcal{F}_{\DDelta,N}^\prime(X_1,\ldots,X_N)$ automatically belongs to  $\mathcal{F}_{\DDelta,N-1}^\prime(X_1,\ldots,X_{N-1}\cup X_N)$ by definition.
	We now focus on proving the inequality $\Pb_N\left(X_1,\ldots,X_N\right)\le\Pb_{N-1}\left(X_1,\ldots,X_{N-1}\cup X_N\right)$.
	We will use the following datum, $\DDelta$, in $\mathcal{F}_{\DDelta,N-1}^\prime(X_1,\ldots,X_{N-1}\cup X_N)$: 
	\begin{itemize}
		\item $\Delta$ is the square $[0,1]\times[0,1]\subset \mathbb{R}^2$.
		\item $p^{}_1 = (0,1)$, $p^{}_{N-1}=(1,1)$, $p^{}_2,\ldots,p^{}_{N-2}\in [0,1]\times \lbrace0\rbrace$.
	\end{itemize}
	Recalling our notation, $\gamma_i$ is the arc along $\partial \Delta$ between $p^{}_i$ and $p^{}_{i+1}$ oriented counterclockwise.
	The proof will follow from a couple of lemmata.
	\begin{lem}\label{lem:pb3pb4_step1}
		Let $\Phi \in \mathcal{F}_{\DDelta,N-1}^\prime(X_1,\ldots,X_{N-1}\cup X_N)$, then for all $\eps>0$ there exists $\widetilde{\Phi}\colon M \to \Delta$ with the following properties:
		\begin{itemize}
			\item $\forall \; 1\le k \le N-2$, $\;\widetilde{\Phi}(X_k)\subseteq \gamma^{}_k.$
			\item There exists $0<\delta=\delta(\eps,\Phi)$ such that
					\begin{align}
					\label{eqn:good_interval1}\widetilde{\Phi}(X_{N-1})\subseteq \gamma^{}_{N-1}\setminus B_{\delt}(p^{}_1), \\
					\label{eqn:good_interval2}\widetilde{\Phi}(X_{N})\subseteq \gamma^{}_{N-1}\setminus B_{\delt}(p^{}_{N-1}).
					\end{align}
			\item $\left\Vert \left\{\widetilde{\Phi}_1,\widetilde{\Phi}_2\right\} \right\Vert \le 
			\left\Vert \left\{\Phi_1, \Phi_2\right\} \right\Vert + \eps$.
		\end{itemize}
	\textnormal{That is, we "push away" the unwanted image of $X^{}_{N_1}$ from a neighborhood of $p^{}_1$, and similarly the unwanted image $X^{}_N$ from a neighborhood of $p^{}_{N-1}$. We do so without increasing the norm of the Poisson bracket too much.}
	\end{lem} 
	
	\begin{proof}
	Let $\eps > 0$.
	We consider the following sets where we would like to alter the values of $\Phi$. Let $\eps_0 < \frac 1 3$. Set:
	\begin{align*}
		\VI &:= \Phi^{-1}\left(B_{\eps_0/2}(p^{}_1)\right)\cap X^{}_{N-1}, \\
		\VNminI &:= \Phi^{-1}\left(B_{\eps_0/2}(p^{}_{N-1})\right)\cap X^{}_{N}.
	\end{align*}
	The notation is chosen to help the reader remember both that $\Phi\left( \VI \right) \subseteq B_{\eps_0/2}(p^{}_1)$ and that $\VI \subseteq X^{}_{N-1}$, and similarly for $\VNminI$.
	\begin{figure}[h]
		\centering
		\includegraphics[scale=0.52]{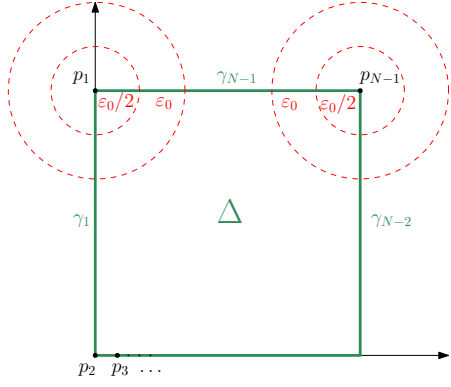}
		\caption{The set $\Delta$ and the balls around the vertices.}
		\label{fig1}
	\end{figure}
	Since $X_1$ and $X_{N-1}$ are closed and disjoint, there exist open sets $\UIp, \UIpp \subset M$ such that:
	\begin{itemize}
		\item $\VI\subseteq \UIp \subset \overline{\UIp} \subset \UIpp$.
		\item There exist an open neighborhood $\Op\left(X_1\right)$ such that ${\overline{\UIpp} \cap \Op\left(X_1\right) = \emptyset}$.
		\item $\Phi\left(\UIpp\right) \subseteq  B_{\eps_0}(p^{}_1)$.
	\end{itemize}
	Fix a smooth cut-off function, $\rho^{}_1\colon M \to [0,1]$, such that $\rho^{}_1|_{\UIp}\equiv 1$ and $\rho^{}_1|_{{\left(\UIpp\right)}^c}\equiv 0$.
	
	Similarly, Since $X_{N},X_{N-2}$ are closed and disjoint, there exist open sets $\UNminIp, \UNminIpp \subset M$ such that:
	\begin{itemize}
		\item $V_{N-1}\subseteq \UNminIp \subset \overline{\UNminIp} \subset \UNminIpp$.
		\item There exist an open neighborhood $\Op\left(X_{N-2}\right)$ such that ${\overline{\UNminIpp} \cap \Op\left(X_{N-2}\right) = \emptyset}$.
		\item $\Phi\left(\UNminIpp\right) \subseteq  B_{\eps_0}(p^{}_{N-1})$.
	\end{itemize}
	Fix a smooth cut-off function, $\rho^{}_{N-1}\colon M \to [0,1]$, such that $\rho^{}_{N-1}|_{\UNminIp}\equiv 1$ and $\rho^{}_{N-1}|_{{\left(\UNminIpp\right)}^c}\equiv 0$.
	
	For any $\delta$ such that $0<\delta<\frac{\eps_0}{2}$ consider $\widetilde{\Phi}_\delta\colon M \to \Delta$ defined by:
	\[
		\widetilde{\Phi}_\delta(x):= \left(\Phi_1(x) + \delt\rho^{}_1(x) - \delt\rho^{}_{N-1}(x),\Phi_2(x)\right).
	\]
	Let us verify the desired properties of $\widetilde{\Phi}_\delta$.

	\begin{clm}\label{clm:1stProperty}
		$\widetilde{\Phi}_\delta\left(X^{}_{N-1}\right) \subseteq B_{\delta}\left(p_1\right)^{c}$ and  $\widetilde{\Phi}_\delta\left(X^{}_{N}\right) \subseteq B_{\delta}(p^{}_{N-1})^{c}$.
	\end{clm}
	\begin{proof}
		We verify the first inclusion as the second is analogous.
		Let $x\in X^{}_{N-1}$. We denote by $\Phi(x) = (a,b)\in\mathbb{R}^2$ its image under $\Phi$.
		Since $\Phi\in \mathcal{F}_{\DDelta,N-1}^\prime(X_1,\ldots,X_{N-1}\cup X_N)$, by definition we have $\Phi\left(X^{}_{N-1}\right) \subseteq \gamma^{}_{N-1}$, so $b=1$ and $0\le a \le 1$. Now let us analyze $\widetilde{\Phi}_\delta(x)$ by cases:
		\[\widetilde{\Phi}_\delta(x) = (a + \delt\rho^{}_1(x) - \delt\rho^{}_{N-1}(x), 1),\]
		\begin{itemize}
		\item If $a \ge \eps_0$ then $a + \delt\rho^{}_1(x) - \delt\rho^{}_{N-1}(x) \ge \eps_0 + 0 - \delta \ge \eps_0 - \eps_0 /2 = \eps_0 /2 > \delta$.
		\item If $a < \eps_0$ then $\Phi(x)\in B_{\eps_0}(p_1)$. Now since $\Phi(\UNminIpp)\subseteq B_{\eps_0}(p^{}_{N-1})$ and since $B_{\eps_0}(p_1) \cap B_{\eps_0}(p^{}_{N-1}) = \emptyset$, it follows that $x\in {U_{N-1}^{\prime\prime}}^{c}$. Recalling that $\rho^{}_{N-1}\vert^{}_{\left({\UNminIpp}\right)^{c}} \equiv 0$ we obtain:
		\[a + \delt\rho^{}_1(x) - \delt\rho^{}_{N-1}(x) = a + \delt\rho^{}_1(x),\]
		and again we argue case by case:
			\begin{itemize}
			\item
			If $\eps_0/2 \le a < \eps_0$ then $a + \delt\rho^{}_1(x) \ge a \ge \eps_0/2 > \delta$.
			\item
			If $a < \eps_0/2$ then $x\in \VI$, therefore $\rho^{}_{1}(x)=1$, thus $a + \delt\rho^{}_1(x) = a+\delta \ge \delta$
			\end{itemize}
		\end{itemize}
		In either case $\widetilde{\Phi}_\delta(x) = (\alpha, 1)\in\mathbb{R}^2$ for some $\alpha\ge \delta$, therefore $\widetilde{\Phi}_\delta(x) \in B_{\delta}\left(p_1\right)^{c}$, hence 	$\widetilde{\Phi}_\delta\left(X^{}_{N-1}\right) \subseteq B_{\delta}\left(p_1\right)^{c}$ completing the proof.
	\end{proof}
	
	\begin{clm}\label{clm:2ndProperty}
			$\widetilde{\Phi}_\delta(X_{N-1}\cup X_{N}) \subseteq \gamma^{}_{N-1}$.
	\end{clm}
	\begin{proof}
		We have to check that $\widetilde{\Phi}_\delta(X_{N-1}\cup X_{N})$ does not contain points lying to the left of $p_1$ or to the right of $p^{}_{N-1}$. We check for $X_{N-1}$ with respect to $p^{}_{N-1}$ as the argument for $x\in X^{}_N$ is analogous. Let $x\in X^{}_{N-1}$. Keeping the notations from the proof of Claim \ref{clm:1stProperty} we have 
			\[\widetilde{\Phi}_\delta(x) = (a + \delt\rho^{}_1(x) - \delt\rho^{}_{N-1}(x), 1).\]
		The argument divides according to the value of $a$.
		\begin{itemize}
			\item If $0\le a \le 1-\eps_0$ then since $\delta < \frac {\eps_0}{2}$, we have:
			\[a + \delt\rho^{}_1(x) - \delt\rho^{}_{N-1}(x) \le 1-\eps_0 + \delta - 0 \le 1 - \eps_0 + \eps_0/2 < 1\]
			\item If $1-\eps_0 < a \le 1$ then $\widetilde{\Phi}_\delta(x)\in B_{\eps_0}(p_{N-1})$. Now since $B_{\eps_0}(p_1) \cap B_{\eps_0}(p^{}_{N-1}) = \emptyset$, and since $\Phi(U^{\prime\prime}_{1})\subset B_{\eps_0}(p^{}_{1})$, it follows that $x\in {U_{1}^{\prime\prime}}^{c}$.
			Recalling that $\rho^{}_{1}\vert^{}_{{U^{\prime\prime}_{1}}^{c}} \equiv 0$ we have:
			\[a + \delt\rho^{}_1(x) - \delt\rho^{}_{N-1}(x) = a - \delt\rho^{}_{N-1}(x) \le 1.\]
		\end{itemize}
		We have shown that $\widetilde{\Phi}_\delta(x) = (\alpha,1)$ where $0\le \alpha \le 1$ thus $x\in \gamma^{}_{N-1}$.	
	\end{proof}
	Combining Claims \ref{clm:1stProperty} and \ref{clm:2ndProperty} we deduce that for all $\delta < \frac {\eps_0}{2}$ Equations (\ref{eqn:good_interval1}) \& (\ref{eqn:good_interval2}) hold.	
	Next we validate:
	\begin{clm}
		\begin{align*}
			\widetilde{\Phi}_\delta(X_1)&=\Phi(X_1) \subseteq \gamma_1, \\
			\widetilde{\Phi}_\delta(X_{N-2})&=\Phi(X_{N-2}) \subseteq \gamma^{}_{N-2}.
		\end{align*}
	\end{clm}
	\begin{proof}
		We verify the claim for $X_1$ as the verification for $X_{N-2}$ is analogous. The claim will follow from $\rho^{}_1|_{\left({\UIpp}\right)^c}\equiv 0$ and $\rho^{}_{N-1}|_{\left({\UNminIpp}\right)^c}\equiv 0$.
		Let $x\in X_1$. Since 
		\[
			\Phi(X_1) \cap B_{\eps_0}(p^{}_{N-1}) \subset \gamma^{}_1 \cap B_{\eps_0}(p^{}_{N-1}) = \emptyset,
		\]
		and since $\Phi(\UNminIpp) \subseteq B_{\eps_0}(p^{}_{N-1})$ we have $X_1 \subseteq \left(\UNminIpp\right)^{c}$, hence ${\rho^{}_{N-1}|_{X_1}\equiv 0}$.
		Now, by definition of $\UIpp$, it satisfies $\overline{\UIpp} \cap X_1 = \emptyset$, so $X_1\subseteq \left({\UIpp}\right)^{c}$, hence also $\rho^{}_1|_{X_1}\equiv 0$.
		Therefore:
		\begin{align*}
			\widetilde{\Phi}_\delta(x)&= \left(\Phi_1(x) + \delt\rho^{}_1(x) - \delt\rho^{}_{N-1}(x),\Phi_2(x)\right) = \\
			&= \left(\Phi_1(x) + 0 - 0,\Phi_2(x)\right) = \left(\Phi_1(x),\Phi_2(x)\right) = \Phi(x).
		\end{align*}	
	\end{proof}
	Last, we verify the following:
	\begin{clm}
		\begin{align*}
		\widetilde{\Phi}_\delta(X_2)&=\Phi(X_2) \subseteq \gamma_2, \\
		\vdots \\
		\widetilde{\Phi}_\delta(X_{N-3})&=\Phi(X_{N-3}) \subseteq \gamma^{}_{N-3}.
		\end{align*}
	\end{clm}
	\begin{proof}
		Let $2\le k \le N-3$ and let $x\in X_k$.
		Since $\Phi\in \mathcal{F}_{\DDelta,N-1}^\prime(X_1,\ldots,X_{N-1}\cup X_N)$, by definition we have $\Phi(x)\in \gamma^{}_k \subseteq [0,1]\times \{0\}$. The segment $[0,1]\times \{0\}$ is disjoint from the union of balls $B_{\eps_0}(p^{}_1)\cup B_{\eps_0}(p^{}_{N-1})$ and since $\Phi\left(\UIpp\right) \subseteq  B_{\eps_0}(p^{}_1)$ and $\Phi\left(\UNminIpp\right) \subseteq  B_{\eps_0}(p^{}_{N-1})$, we deduce that $x\in \left(\UIpp\right)^{c} \cap \left(\UNminIpp\right)^{c}$. Recalling that $\rho^{}_1|_{\left({\UIpp}\right)^c}\equiv 0$ and $\rho^{}_{N-1}|_{\left({\UNminIpp}\right)^c}\equiv 0$ we compute:
		\begin{align*}
		\widetilde{\Phi}_\delta(x)&= \left(\Phi_1(x) + \delt\rho^{}_1(x) - \delt\rho^{}_{N-1}(x),\Phi_2(x)\right) = \\
		&= \left(\Phi_1(x) + 0 - 0,\Phi_2(x)\right) = \left(\Phi_1(x),\Phi_2(x)\right) = \Phi(x).
		\end{align*}
	\end{proof}
	To conclude the proof we compute $\left\Vert \left\{\widetilde{\Phi}_{\delta,1},\widetilde{\Phi}_{\delta,2}\right\} \right\Vert$:
	\begin{align*}
		\left\Vert \left\{\widetilde{\Phi}_{\delta,1},\widetilde{\Phi}_{\delta,2}\right\} \right\Vert &=
		\left\Vert \left\{\Phi_1 + \delt\rho^{}_1(x) - \delt\rho^{}_{N-1}(x),\Phi_2\right\} \right\Vert \le \\
		&\le \left\Vert \left\{\Phi_1,\Phi_2\right\} \right\Vert + \delt\left\Vert \left\{\rho^{}_1,\Phi_2\right\} \right\Vert + \delt\left\Vert \left\{\rho^{}_{N-1},\Phi_2\right\} \right\Vert \overset{\delt\to 0}{\longrightarrow} \left\Vert \left\{\Phi_1,\Phi_2\right\} \right\Vert.
	\end{align*}
	
	Thus the lemma is proven by picking $\widetilde{\Phi} := \widetilde{\Phi}_{\delta_0}$ for
	\[\delta_0 < \min\left\lbrace\frac{\eps_0}{2}, \frac{\eps}{2}{\left(\left\Vert \left\{\rho^{}_1,\Phi_2\right\} \right\Vert + \left\Vert \left\{\rho^{}_{N-1},\Phi_2\right\} \right\Vert\right)}^{-1}\right\rbrace.\]
	\end{proof}
	Next we prove:
	\begin{lem}\label{lem:pb3pb4_step2}
		Let $\widetilde{\Phi} := \widetilde{\Phi}_{\delta_0}$ obtained from Lemma \ref{lem:pb3pb4_step1}, then for all $\eps>0$ there exists $\widehat{\Phi}\colon M \to \Delta$ with the following properties:
		\begin{itemize}
			\item There exist $N$ points on $\partial \Delta$, denoted $p^\prime_1,\ldots,p^\prime_{N}$, defining arcs, $\gamma^\prime_1,\ldots, \gamma^\prime_N$, such that for all $1\le k \le N$, $\widehat{\Phi}(X_k)\subseteq \gamma^{}_k$.
			\item $\left\Vert \left\{\widehat{\Phi}_1,\widehat{\Phi}_2\right\} \right\Vert \le \frac{1+\eps}{1-\eps}\left\Vert\left\{\widetilde{\Phi}_1,\widetilde{\Phi}_2\right\}\right\Vert$.
		\end{itemize}	
	\end{lem} 		
	\begin{proof}
	The strategy of the proof is to compose $\widetilde{\Phi}$ with a pseudoretract onto a square that maps the segment $\gamma^{}_{N-1}\setminus \left(B_\delt(p^{}_1)\cup B_\delt(p^{}_{N-1})\right)$ to a vertex, which is going to be the new point, $p^{}_N$. Namely, we seek to contract to a point the problematic segment where the overlap of $\widetilde{\Phi}\left(X^{}_N\right)$ and $\widetilde{\Phi}\left(X^{}_{N-1}\right)$ occurs. The set $X^{}_N$ is then mapped to the left of $p^{}_N$ and $X^{}_{N-1}$ is mapped to the right of $p^{}_N$, while we maintain control on how much the Poisson bracket is increased.
	Recall Remark \ref{rem:pseudoretract_property}; The pseudoretract of $\mathbb{R}^2$ onto a square, described in Proposition \ref{prop:pseudoretractsExist} has the property of mapping a sector spanned by the opposite angle to the interior angle at a vertex to that vertex. See Figure \ref{fig2}.
	
	Let $\eps>0$ and consider a square of side length $1-\eps$ which we denote by
	\[\Delta_2:=
	[0,1-\eps]\times[0,1-\eps].\] 
	Let $S$ be a symplectomorphism mapping $\Delta$ to a subset $S(\Delta)\subset \mathbb{R}^2$ such that
	\begin{itemize}
		\item $\Delta_2$ is contained in the interior of $S\left(\Delta\right)$.
		\item The arc $S\left(\gamma^{}_{N-1}\setminus \left(B_\delt(p^{}_1)\cup B_\delt(p^{}_{N-1})\right)\right)$ lies inside the sector $A_v$ spanned by the opposing angle to the interior angle at the vertex $(1-\eps,1-\eps)$, with its boundary points lying on the line extensions of the edges of $\Delta_2$ adjacent to $v$.
	\end{itemize}
	\begin{figure}[h]
			\centering
			\includegraphics[scale=0.62]{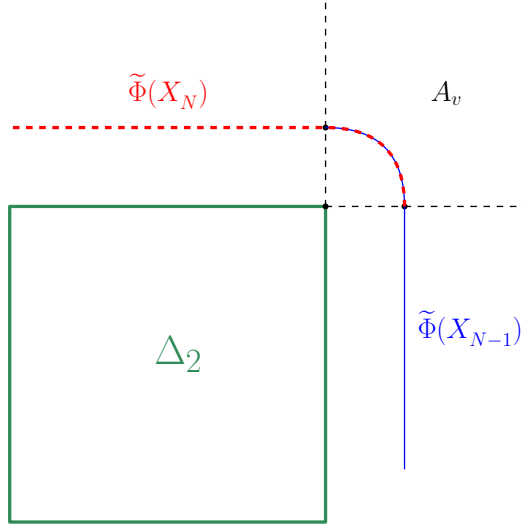}
			\caption{The configuration of the arcs and the angle $A_v$.}
			\label{fig2}
	\end{figure}
	Let $T$ be an $\eps$-pseudoretract onto $\Delta_2$, and $\mathcal{H}_{\frac {1}{1-\eps}}$ the homothety by a factor of $\frac {1}{\sqrt{1-\eps}}$,
	Define
	\[
		\widehat{\Phi}:= \mathcal{H}_{\frac {1}{1-\eps}}\circ T \circ S\circ\widetilde{\Phi}.
	\]
	Then
	\begin{align*}
		\left\Vert \left\{\widehat{\Phi}_1,\widehat{\Phi}_2\right\} \right\Vert\le
		\frac {1+\eps}{1-\eps}\left\Vert \left\{S\circ\widetilde{\Phi}_1,S\circ\widetilde{\Phi}_2\right\} \right\Vert\le
		\frac {1+\eps}{1-\eps}\left\Vert \left\{\widetilde{\Phi}_1,\widetilde{\Phi}_2\right\} \right\Vert \overset{\eps\to 0}{\longrightarrow} \left\Vert\left\{\widetilde{\Phi}_1,\widetilde{\Phi}_2\right\} \right\Vert.
	\end{align*}
	
	We then define the points $p^\prime_i$ by $p^\prime_i := \mathcal{H}_{\frac{1}{1-\eps}}\circ T\circ S (p_i)\in \partial \Delta_2$ for $1\le i \le N-1$ and $p^\prime_N := v\in \partial \Delta_2$. By our construction, for all $1\le k \le N$ it holds that $\widehat{\Phi}(X_k)\subseteq \gamma^{}_k$, concluding the proof.
	\end{proof}
	To conclude the theorem's proof, WLOG one can assume that $\Phi$ satisfies (CS) with respect to a point $p\in \Delta \setminus \left(B_{\delt}(p^{}_{1}) \cup B_{\delt}(p^{}_{N-1})\right)$ otherwise use the methods of Proposition \ref{prop:pb_invariance} to move $p$ outside these balls. Lemmata \ref{lem:pb3pb4_step1} and \ref{lem:pb3pb4_step2} alter $\Phi$ by post-compositions, so (CS) condition is preserved. Hence, the function $\widehat{\Phi}$ constructed in Lemma \ref{lem:pb3pb4_step2} is admissible for $\Pb_{N}^{\DDelta_2}$ where $\DDelta_2 = ( [0,1] \times [0,1],\{p^\prime_1,\ldots,p^\prime_N\})$, and by choosing $\delta \to 0$ and $\eps \to 0$ small enough, the theorem follows.

\end{proof}
\subsection{Proof of Theorem \ref{thm:pb3 limit}}
The invariant $\Pb^{}_X$ satisfies monotonicity and semi-continuity properties similarly to $\pb_n$ in \cite{BEP}.
\begin{prop}\label{prop:PbaMonotone}
	\textbf{Monotonicity:} Let $X,Y$ be two compact sets such that $X\subseteq Y$.Denote by $i\colon X \hookrightarrow Y$ the inclusion map. Then for any class ${\alpha \in H^1(Y;\mathbb{Z})}$ we have:
	\[\Pb^{}_X(i^*\alpha)  \le \Pb^{}_Y(\alpha).\]
\end{prop}
\begin{proof}
	Any function $\Phi\colon M \to \overline{B_1}$ admissible for the set over which we infimize in $\Pb^{}_Y(\alpha)$ is also admissible for $\Pb^{}_X(i^*\alpha)$
\end{proof}

\begin{prop}\label{prop:PbaSemi}
	\textbf{Semicontinuity:} Let $X$ be a compact subset of a symplectic manifold $M$.
	Fix a class $\alpha \in H^1(X;\mathbb{Z})$ and consider an extension of it to a neighborhood, $U^{}_X$, of $X$, denoted by $\bar{\alpha} \in H^1(U^{}_X;\mathbb{Z})$, which exists by Proposition \ref{prop:rhoIsom}.
	Let $X_n$ be a sequence of compact sets contained in $U_X$, converging to $X$ in the Hausdorff distance.
	The class $\bar{\alpha}$ determines a class in $H^1({X_n};\mathbb{Z})$ by pullback along the inclusion $X_n\hookrightarrow U^{}_X$, which we denote by $\bar{\alpha}\vert^{}_{X_n}$.
	Then:
	\[\limsup_{n\to\infty}\; \Pb^{}_{X_n}(\bar{\alpha}\vert^{}_{X_n}) \le \Pb^{}_{X}(\alpha).\]
\end{prop}

\begin{proof}
	For any function $\Phi\colon M \to \overline{B_1}$ admissible for the set over which we infimize in $\Pb^{}_X(\alpha)$ there exists $N$ such that for all $n\ge N$, $\Phi$ is also admissible for $\Pb^{}_{X_n}(\bar{\alpha})$. This is because there exists a neighborhood of $X$, $U_\Phi$, such that $\Phi(U_\Phi)\subset S^1$ and $[\Phi\vert^{}_{U_\Phi}] = \bar{\alpha}\vert^{}_{U_\Phi}$ (by Proposition \ref{prop:rhoIsom}) and there exists $N$ such that for all $n>N$, $X_n\subset U_\Phi$.
\end{proof}

\begin{cor}\label{cor:PbaMonotoneLimit}
	Let $X$ be compact set in a symplectic manifold $M$, and let $X_n$ be a monotone decreasing sequence of compact sets (namely $X_{n+1}\subseteq X_n$), containing $X$, converging in the Hausdorff metric to $X$. Fix  $\alpha\in H^1(X;\mathbb{Z})$ and consider an extension of it to a neighborhood $U^{}_X$ of $X$, denoted by $\bar{\alpha} \in H^1(U^{}_X;\mathbb{Z})$, which exists by Proposition \ref{prop:rhoIsom}.
	Then
	\[\lim_{n\to\infty} \Pb_{X_n} (\bar{\alpha}\vert^{}_{X_n}) = \Pb_X(\alpha).\]
\end{cor}
\begin{proof}
	By the monotonicity property we have that $\Pb_{X_n}(\bar{\alpha}\vert^{}_{X_n})$ is a monotone decreasing sequence of numbers bounded below by $\Pb_{X}(\bar{\alpha}\vert^{}_{X}) = \Pb_{X}(\alpha)$, therefore it converges. On the other hand we have from semi-continuity that
	\[\lim_{n\to\infty} \Pb^{}_{X_n}(\bar{\alpha}\vert^{}_{X_n}) = \limsup_{n\to\infty}\; \Pb^{}_{X_n}(\bar{\alpha}\vert^{}_{X_n}) \le \Pb^{}_{X}(\alpha).\]
	Completing the proof.
\end{proof}

For brevity and to avoid cumbersome notation we describe the proof of Theorem \ref{thm:pb3 limit} for $N=3$, the proof for any $N$ is similar. We prove the following:
\begin{prop}
	Let $X_1,X_2,X_3$ be a triple of compact subsets in a symplectic manifold, $M$, such that ${X_1\cap X_2 \cap X_3=\emptyset}$. Let $K_n$ be a decreasing sequence of compact neighborhoods of $X_1 \cap X_3$ converging to $X_1 \cap X_3$ in the Hausdorff distance, and moreover assume that $K_1\cap X_2=\emptyset$. Then:
	\begin{enumerate}
		\item 
		The following limit exists:
		\[
		\lim_{K_n\searrow X_1\cap X_3} \Pb_4(\overline{X_1 \setminus K_n},X_2,\overline{X_3\setminus K_n}, K_n).
		\]
		\item 
		$\displaystyle\lim_{K_n\searrow X_1\cap X_3} \Pb_4(\overline{X_1 \setminus K_n},X_2,\overline{X_3\setminus K_n}, K_n) = \Pb_3(X_1,X_2,X_3)$.
	\end{enumerate} 
\end{prop}
\begin{rem}
	Ideally one would like to take $\overline{X_1\setminus X_3}, X_2, \overline{X_3\setminus X_1}, X_1\cap X_3$ as the quadruple of sets in $\Pb_4$  in the proposition, but this quadruple might not satisfy $\left(\overline{X_1\setminus X_3}\right) \cap \left(\overline{X_3\setminus X_1}\right) = \emptyset$. Therefore we have to approximate $ X_1\cap X_3$ from outside by compact neighborhoods.
\end{rem}
\begin{rem}
	The sequences of sets, $\overline{X_1 \setminus K_n}$ and $\overline{X_3 \setminus K_n}$, are monotone increasing, and the sequence $K_n$ is monotone decreasing, thus one cannot directly apply \cite{BEP}'s monotonicity statement for this quadruple. Nevertheless the union $Z_n = X_1 \cup X_2 \cup X_3 \cup K_n$ is indeed monotone decreasing so we can proceed with monotonicity of $\Pb^{}_{Z_n}$:
\end{rem}
\begin{proof}
	By Theorem \ref{thm:actualThm1}:
	\[\Pb_4(\overline{X_1 \setminus K_n}, X_2,\overline{X_3\setminus K_n}, K_n) = \Pb_3(\overline{X_1 \setminus K_n},X_2,\overline{X_3\setminus K_n}\cup K_n),\]
	and by Theorem \ref{thm:htpyCharacterization}:
	\[\Pb_3(\overline{X_1 \setminus K_n},X_2,\overline{X_3\setminus K_n}\cup K_n) = \Pb^{}_{Z_n}([f_n]),\]
	where $Z_n := X_1 \cup X_2 \cup X_3 \cup K_n$ and $f_n$ is a function ,$f_n:Z_n \to S^1$, such that:

	\[f_n\left(\overline{X_1 \setminus K_n}\right) \subseteq \gamma_1,\quad f_n\left(X_2\right) \subseteq \gamma_2, \quad f_n\left(\overline{X_3\setminus K_n}\cup K_n\right) \subseteq \gamma_3.\]
	We note that $\overline{X_3\setminus K_n}\cup K_n = X_3\cup K_n$ and that we can choose $f_n = g\vert^{}_{Z_n}$ where $g:Z_1\to S^1$ is a function such that 
	\[\forall i,\: g\left(X_i\right)\subseteq \gamma_i\text{ and  }g\left(K_1\right) = \gamma_1\cup \gamma_3.\] 
	Therefore $\Pb^{}_{Z_n}([f_n]) = \Pb^{}_{Z_n}([g\vert^{}_{Z_n}])$ and by Corollary \ref{cor:PbaMonotoneLimit}:
	\[\lim_{n\to\infty} \Pb^{}_{Z_n}([g\vert^{}_{Z_n}]) = \Pb^{}_{Z}([g\vert^{}_{Z}]).\]
	Where $Z = X_1\cup X_2 \cup X_3$.
	Finally, by Theorem \ref{thm:htpyCharacterization}, $\Pb^{}_{Z}([g\vert^{}_{Z}]) = \Pb_3(X_1,X_2,X_3)$.
\end{proof}

\subsection{$\pb_3$ and Dynamics}

The above proposition, expressing $\Pb_3$ a limit of $\Pb_4$s, yields a dynamical interpretation of $\pb_3(X_1,X_2,X_3)$ in terms of Hamiltonian chords connecting $X_1\setminus X_3$ and $X_3\setminus X_1$ for flows of functions which are bounded below by $1$ near $X_2$ and bounded above by $0$ near $X_3\cap X_1$, in a similar fashion to the dynamical interpretation given for $\pb_4$ in \cite{BEP}. Recall that $\nicefrac{1}{\pb_4(X_0, X_1, Y_0, Y_1)}$ has the following dynamical interpretation (Note that in $\pb_4$ we do not use the cyclical notation for the sets): 
\begin{thm*}[\cite{BEP} 1.10]
	Let $X_0, X_1, Y_0, Y_1 \subset M$ be a quadruple of compact sets such that $X_0 \cap X_1 = Y_0 \cap Y_1 = \emptyset$ and $\nicefrac{1}{\pb4(X_0, X_1, Y_0, Y_1)} = p > 0$.
	Let $G \in C^{\infty}_c(M)$ be a Hamiltonian
	with $G\vert^{}_{Y_0} \le  0 $ and $G\vert^{}_{Y_1} \ge  1$ generating a Hamiltonian flow $g_t$. Then there exists a Hamiltonian chord of time length $\le p$ going from $X_1$ to $X_0$ or sfrom $X_0$ to $X_1$.
\end{thm*} 
We now show the following (slightly weaker due to noncompactness of $X_3\setminus X_1$ and $X_1\setminus X_3$) analogue for $\pb_3$:
\begin{cor}\label{cor:pb3_dynamics}
	Let $X_1,X_2,X_3$ be a triple of compact subsets in a symplectic manifold, $M$, such that ${X_1\cap X_2 \cap X_3=\emptyset}$. Let $G\in C^{\infty}_{c}$ be a Hamiltonian $G \colon M \to \mathbb{R}$ such that
	$G^{}\vert_{X_2} \le 0$ and $G\vert^{}_{X_1\cap X_3} \ge 1$.
	Assume ${p_0 := \frac{1}{2\pb_3(X_1, X_2, X_3)} > 0}$
	Then for all $p>p_0$ there exists a trajectory of the Hamiltonian flow of G of time-length  $\le p$ going
	from $X_3\setminus X_1$ to $X_1\setminus X_3$ or from $X_1\setminus X_3$ to $X_3\setminus X_1$.
\end{cor}
\begin{proof}
	For every $\eps>0$ there exists $K_{\eps}$ compact such that $X_1\cap X_3 \subset K_{\eps}$ and 
	\[
	\nicefrac{1}{\left(2\pb_3(X_1,X_2,X_3)\right)} -\eps \le \nicefrac{1}{\pb_4(\overline{X_1 \setminus K_{\eps}},X_2,\overline{X_3\setminus K_{\eps}}, K_{\eps})} \le \nicefrac{1}{\left(2\pb_3(X_1,X_2,X_3)\right)}.
	\]
	Pick $\eps>0$ such that $ \nicefrac{1}{\pb_4(\overline{X_1 \setminus K_{\eps}},X_2,\overline{X_3\setminus K_{\eps}}, K_{\eps})} \ge p_0-\eps > 0$. For any $\delta>0$ there exists a compact set $K_{\eps,\delta}$ such that $X_1\cap X_3 \subseteq K_{\eps,\delta}\subseteq K_{\eps}$ and $G\vert^{}_{K_{\eps,\delta}} \ge 1-\delta$.
	We have that:
	\[
	0 < \nicefrac{1}{\pb_4(\overline{X_1 \setminus K_{\eps}},X_2,\overline{X_3\setminus K_{\eps}}, K_{\eps})} \le \nicefrac{1}{\pb_4(\overline{X_1 \setminus K_{\eps,\delta}},X_2,\overline{X_3\setminus K_{\eps,\delta}}, K_{\eps,\delta})} \le p_0.
	\]
	Consider $\frac{G}{1-\delta}$, it is a Hamiltonian such that $\frac{G}{1-\delta}^{}\big\vert_{X_2} \le 0$ and $\frac{G}{1-\delta}\big\vert^{}_{K_{\eps,\delta}} \ge 1$, and thus from the positivity of $\pb_4(\overline{X_1 \setminus K_{\eps,\delta}},X_2,\overline{X_3\setminus K_{\eps,\delta}}, K_{\eps,\delta})$ there exists a Hamiltonian chord of $\frac{G}{1-\delta}$ connecting $\overline{X_1 \setminus K_{\eps,\delta}}$ and $\overline{X_3\setminus K_{\eps,\delta}}$ (in some direction) with time length $\le p_0$, 
	Therefore by rescaling we get a chord of $G$ connecting the same sets with time length $\le \frac{p_0}{1-\delta}$. Picking $\delta$ small enough such that $\frac{p_0}{1-\delta} < p$ finishes the proof.
\end{proof}

\section{The invariant $\Pb_X(\alpha)$}	
\subsection{Setup}
For a topological space $X$, denote by $[X:S^1]$ the set of homotopy classes of continuous maps from $X$ to $S^1$.
If $Z\subset X$ is a subspace, then restriction (of functions and of homotopies) induces a map $[X:S^1]\to[Z:S^1]$.
Given a compact subset $X\subset M$ of a manifold $M$ we define: 
\[{\mathcal{N}H^1(X) := \displaystyle\varinjlim_{U\supseteq X} [U:S^1]}. \]
Where the limit is taken on the directed system of open sets $U$ containing $X$.
The notation $H^1$ is suggestive of the well known isomorphism $H^1(X;\mathbb{Z})\cong[X:S^1]$, due to $S^1$ being a $K(\mathbb Z, 1)$ space, and where the cohomology is \v{C}ech cohomology of the constant sheaf $\mathbb{Z}$. (The isomorphism is proven in \cite{Morita1975}).
$\mathcal{N}$ stands for $\mathcal{N}\!$eighborhood. Restriction of maps to $X$ induces a map $\rho \colon \mathcal{N}H^1(X) \to [X:S^1]$. Moreover, in light of the isomorphism with cohomology, the sets $[U:S^1]$ admit a group structure, and all the maps induced by restriction to subsets are in fact group homomorphisms. Moreover, since $S^1$ is a topological group, the group structure on $[U,S^1]$ is induced from the group structure on $S^1$. See chapter 22 in \cite{May} for details on Eilenberg-MacLane spaces and their relation to cohomology.

\begin{prop}\label{prop:rhoIsom}
	The map $\rho$ defined above is surjective and injective, i.e. an isomorphism of groups.
\end{prop}
\begin{proof}
	We begin with \textbf{surjectivity} of $\rho$. Consdier $S^1$ embedded in $R^2$ as the unit circle.
	Let $[\varphi]\in[X:S^1]$, that is $\varphi\colon X\to S^1\subset \mathbb{R}^2$, and denote $\varphi(x) = (\varphi_1(x),\varphi_2(x))$ where each $\varphi_i$ is a map $\varphi_i\colon X \to \mathbb{R}$.
	Since $X$ is closed in $M$, by the Tietze extension theorem, each $\varphi_i$ extends to a continuous function $\widetilde{\varphi_i}\colon M \to \mathbb{R}$, which together define $\widetilde{\varphi} \colon M \to \mathbb{R}^2$.
	Let $B_{1/2}\subset \mathbb{R}^2$ be an open ball of radius $\frac{1}{2}$ centered at the origin. There exists a retract $\psi\colon R^2\setminus B \to S^1$.
	Consider $M\setminus\overline{\varphi^{-1}(B)}$ which is an open neighborhood of $X$, then $\psi \circ \widetilde{\varphi} \colon M\setminus\overline{\varphi^{-1}(B)}\to S^1$ is an extension of $\varphi$ to an open neighborhood of $X$, inducing an element $[\psi\circ\varphi] \in \mathcal{N}H^1(X)$ such that $\rho\left([\psi\circ\widetilde{\varphi}]\right) = [\varphi]$.\\
	Next we prove \textbf{injectivity} of $\rho$. We need to show the following:
	Let $U$ be an open neighborhood of $X$ and let $\varphi_0, \varphi_1:U\to S^1$, assume $\rho([\varphi_0])=\rho([\varphi_1])$, i.e. there exists a homotopy $F_t\colon X\times [0,1] \to S^1$ such that $F_0=\varphi_0|_X$ and $F_1=\varphi_1|_X$. We have to show that there exists an open set $V$ such that $X\subseteq V \subseteq U$, and a homotopy $G_t\colon V\times[0,1] \to S^1$ such that $G_0=\varphi_0|^{}_V$ and $G_1=\varphi_1|^{}_V$.
	The argument for existence of such a homotopy is similar to the argument showing surjectivity.
	WLOG, assume $\overline{U}$ is compact. Pick an open set $W\subset U$ such that $\overline{W} \subset U$.
	Consider the following closed subset of $M\times [0,1]$:
	\[
		Z:= \left(\overline{W} \times \{0\}\right) \cup \left(X\times[0,1]\right)  \cup \left(\overline{W} \times \{1\}\right).
	\]
	Consider the function $\psi\colon Z \to S^1$ defined by:
	\begin{align*}
	\psi(x,t):= 
		\begin{cases*}
		\varphi_0(x) & 	if $t=0$ \\
		\varphi_1(x) &     if $t=1$ \\
		F_t(x) & otherwise.
		\end{cases*}		
	\end{align*} 
	 Repeating the argument used in the surjectivity part, this time for the  the compact subset $Z\subset M\times [0,1]$ and $\psi\colon Z \to S^1$, in the manifold with boundary $M\times[0,1]$ yields an extension of $\psi$, denoted by $\widetilde{\psi}:N\to S^1$, where $N$ is some neighborhood of $Z$ in $M\times[0,1]$. By compactness of $Z$, the set $N$ contains an open set of the form $V\times[0,1]$ where $V$ is open in $M$. The restriction $\widetilde{\psi}|^{}_V$ provides the desired homotopy between $\varphi_0|^{}_V$ and $\varphi_1|^{}_V$.
\end{proof}

\begin{clm}\label{clm:pb3_nonempty}
	Let $M$ be a manifold and $X$ a compact set such that $X=X_1\cup X_2 \cup X_3$, with each $X_k$ compact such that $X_1\cap X_2 \cap X_3 = \emptyset$. Then, there exists a continuous $f\colon X \to S^1$ such that for all $1\le k \le 3$, $f(X_k)\subseteq \gamma^{}_k$ where $\gamma^{}_k=\left\lbrace e^{i\theta} \,\middle|\, \theta\in \left[\frac{2\pi (k-1)}{3},\frac{2\pi k}{3}\right] \right\rbrace$. In fact we can choose $f$ such that it extends to a neighborhood, $\mathcal {O}p\left( X_k \right)$, of each $X_k$ and satisfies $f\left(\mathcal {O}p\left( X_k \right)\right) \subseteq \gamma^{}_k$.
\end{clm}
\begin{proof}
	This argument essentially appears in \cite{BEP}, showing that the set over which we infimize in $\pb_3$ is not empty. Consider the open cover of $M$ given by $\left(M\setminus X_1,M\setminus X_2,M\setminus X_3\right)$ and let $\rho^{}_i$ be a partition of unity subordinate to that cover. Consider $f=(\rho_1\vert^{}_X,\rho_2\vert^{}_X)\colon X\to \Delta$ where $\Delta$ is the boundary of a right triangle whose vertices are $(0,0), (1,0), (0,1)$. The result is obtained by composing $f$ with a homeomorphism from $\Delta$ to $S^1$. In fact, by composing $f$ with a pesudoretract on a smaller triangle first, we get $f\left(\mathcal {O}p\left( X_k \right)\right) \subseteq \gamma^{}_k$
\end{proof}	

\begin{clm}
	Let $X=X_1\cup X_2 \cup X_3$ with $X_k$ compact such that $X_1\cap X_2 \cap X_3 = \emptyset$. Then any two functions $f,g\colon X \to S^1$ such that for all $1\le k \le 3$, $f(X_k),g(X_k)\subseteq \gamma^{}_k$ are homotopic.
\end{clm}
\begin{proof}
	Identify $\gamma_k$ with $[0,1]$ via a homeomorphism $\sigma_k \colon \gamma_k \to [0,1]$. For each $k$ we homotope between $\sigma_k\circ f\vert^{}_{X_k}$ and $\sigma_k\circ g\vert^{}_{X_k}$ by the linear homotopy:
	\[h_t :=\sigma_k^{-1}\circ\left((1-t)\sigma_k\circ f + t\sigma_k\circ g\right).\]
	Note that $f(X_k\cap X_{k+1})=g(X_k\cap X_{k+1})$ which equals the far (counterclockwise) endpoint of $\gamma_k$ (addition is to be taken cyclically). Moreover, by linearity this also holds for $h_t$, for all $t\in[0,1]$. Hence, we can homotope between $f$ and $g$ over each $X^{}_k$ sequentially.
\end{proof}
We summarize the contents of the above claims in the following corollary:
\begin{cor}\label{cor:decompositionClass} Any decomposition $X=X_1\cup X_2 \cup X_3$ such that ${X_1\cap X_2 \cap X_3 = \emptyset}$ determines a class $\alpha\in H^1(X;\mathbb{Z})$, hence a class in $\alpha\in\mathcal{N}H^1(X)$. The class $\alpha$ is defined by picking any function $f\colon X\to S^1$ such that $f(X_k)\subseteq \gamma^{}_k$ for $1\le k \le 3$ and setting $\alpha=[f]$. 
\end{cor}
\begin{defn}
	For any $\alpha\in H^1(X;\mathbb{Z})$ define :
	\[Pb^{}_X(\alpha):=\inf\left\lbrace 
	\left\Vert \left\lbrace \phi_1, \phi_2 \right\rbrace \right\Vert
	\,\middle|\,
	\begin{tabular}{@{}l@{}}
	$\Phi=(\phi_1,\phi_2)\colon M\to \overline{B_1}$ such that \\ $\phi_1,\phi_2$ have compact support, and \\
	$\exists \underset{\text{open}}{U}\supset X,\ \phi\vert^{}_U\subseteq S^1,\  \left[ \phi\vert^{}_X\right]=\alpha$
	\end{tabular}
	\right\rbrace.\]
\end{defn}
\begin{rem}
	While writing the paper the author learned that a similar definition of this sort (using $\frac{\Phi^*\omega_{\mathbb{R}^2}}{\omega_{\mathbb{R}^2}}$, over a similar class of functions) was suggested years ago by Frol Zapolsky.
\end{rem}

\subsection{Proof of Theorem \ref{thm:htpyCharacterization}}
\begin{thm}
	Let $X$ be be a compact subset of a symplectic manifold $M$, and assume $X=X_1\cup X_2 \cup X_3$ where $X_1\cap X_2 \cap X_3 = \emptyset$ and each $X_k$ is compact. Denote by $\alpha\in H^1(X;\mathbb{Z})$ the class determined by the decomposition $X=X_1\cup X_2 \cup X_3$ as in Corollary \ref{cor:decompositionClass}. Then:
	\[\Pb_3(X_1,X_2,X_3)=\Pb^{}_X(\alpha).\]
\end{thm}
\begin{proof}
	We start by showing $\Pb_3(X_1,X_2,X_3)\ge \Pb^{}_X(\alpha)$.
	Any $\Phi\colon M \to \Delta$ admissible for $\Pb_3$ can be made into a map admissible for $\Pb_X(\alpha)$ with an arbitrary $\eps$-increase of the norm of the Poisson bracket by composing with a smooth map from the triangle to a disc as done in the proof of Claim \ref{prop:pb_invariance}.
	
	We now turn to proving $\Pb^{}_X(\alpha) \ge Pb_3(X_1,X_2,X_3)$. 
	Let $\Phi\colon M \to \overline{B_1}$ admissible for $\Pb_X(\alpha)$, then there exists an open set $U\supset X$ such that $\Phi(U) \subseteq S^1$.
	Shrinking $U$ if necessary, it follows from Claims \ref{prop:rhoIsom} and \ref{clm:pb3_nonempty} that there exist open sets $U_0, U_1, U_2$ such that:
	\begin{enumerate}
		\item $U_k\supset X_k$ for $1\le k \le 3$.
		\item  $U=U_0\cup U_1 \cup U_2$.
		\item  There exists a function $f\colon U \to S^1$ such that for all $1\le k \le 3$, $f\left( U_k \right) \subseteq \gamma^{}_k$.
		\item  $\Phi\vert^{}_U \overset{htpy}{\sim}f$.
	\end{enumerate}
	We denote the homotopy in (4.) by $f_t$, so and $f_0 = \Phi\vert^{}_U$ and $f_1=f$. By the smooth approximation theorem (Whitney's approximation), $f_t$ can be chosen to be smooth.
	
	Let $\eps>0$. We will construct a smooth function $\widehat{\Phi}\colon M \to \overline{B_1}$ such that:
	\begin{enumerate}
		\item $\widehat{\Phi}^{-1}\left(B_{1-3\eps}\right) = \Phi^{-1}\left(B_{1-3\eps}\right)$.
		\item $\widehat{\Phi}\vert^{}_{\Phi^{-1}\left(B_{1-3\eps}\right)} = \Phi\vert^{}_{\Phi^{-1}\left(B_{1-3\eps}\right)}$.
		\item $\widehat{\Phi}\vert^{}_U = f$.
	\end{enumerate}
	From it, we obtain by a composition with a pseudoretract onto $\overline{B_{1-3\eps}}$ a function with the desired bounds on the norm of the Poisson bracket and the desired behaviour on $X$, by application of Corollary. \ref{cor:pb_bound_retract}

	To construct $\widehat{\Phi}$ we will need the the definition and characterization of a cofibration (Sometimes called a Borsuk pair), see \cite{May} for a deeper treatment.
	\begin{defn}
		A continous map $i\colon Z\to X$ is called a \textbf{cofibration} if it satisfies the homotopy extension property with respect to all spaces $Y$, that is, if for every homotopy $f_t\colon Z\times[0,1] \to Y$ and every map $F_0\colon X\to Y$ extending $f_0$, namely, $F_0|^{}_Z=f_0$, there exists an extension of $f_t$ to a homotopy $F_t\colon X\times[0,1]\to Y$, such that $F_t|^{}_Z=f_t$.
		In a diagram:
		\[
		\begin{tikzcd}
		Z \arrow[hookrightarrow,r, "\operatorname{id}\times\{0\}"] \arrow[hookrightarrow,d, "i"]
		& {Z\times[0,1]} \arrow[d, "f_t"]  \arrow[hookrightarrow,ddr, bend left, "i\times \operatorname{id}"]\\ 
		X \arrow[r, "F_0"]  \arrow[hookrightarrow,drr, bend right, "\operatorname{id}\times\{0\}"']
		& Y \\ & & 
		\arrow[ul, dotted, "\exists F_t"'] X\times[0,1]
		\end{tikzcd}
		\]
	\end{defn}
	The next lemma, whose proof we postpone to the end of the section, states that for a closed subset, $X$, of a manifold, $M$, there is an arbitrarily small "thickening" such that the inclusion of the thickened neighborhood into $M$ is a cofibration.
	\begin{lem}\label{lem:cofib_thicken}
		Let $M$ be a manifold, and let $X\subset U\subset M$, such that $X$ is closed and $U$ is open in $M$. Then there exists a closed neighborhood, $Z$, of $X$, such that $X\subset Z \subset U$, and such that the inclusion map $Z\hookrightarrow M$ is a cofibration.
	\end{lem}
	We continue with the proof of Theorem \ref{thm:htpyCharacterization}.
	Denote by $A^{1}_{1-2\eps} := \overline{B_1} \setminus \overline{B_{1-2\eps}}$ the annulus of radii $1$ and $1-2\eps$. $A^{1}_{1-2\eps}$ is open in the topology of $\overline{B_1}$.
	We proceed in several steps, in each we modify the function constructed in the previous, culminating in the desired function $\Psi$ satisfying what is needed. Figure \ref{fig3} depicts all the balls involved in the construction.
	\\ \textbf{Step 1:} We construct $\widetilde{\Phi}\colon \Phi^{-1}(A^{1}_{1-2\eps}) \to A^{1}_{1-2\eps}$ such that: 
	\begin{enumerate}
		\item $\widetilde{\Phi}\vert^{}_{U} \equiv f$.
		\item $\widetilde{\Phi}\vert^{}_{\Phi^{-1}(\overline{B_{1-\eps}}\setminus B_{1-3\eps/2})}\equiv \Phi\vert^{}_{\Phi^{-1}(\overline{B_{1-\eps}}\setminus B_{1-3\eps/2})}$.
	\end{enumerate}
	Set:
	\begin{align*}
	Q&:= X \cup \Phi^{-1}(\overline{B_{1-\eps}}\setminus B_{1-3\eps/2}). \\
	W&:= U \cup \mathcal N\left(\Phi^{-1}(\overline{B_{1-\eps}}\setminus B_{1-3\eps/2})\right).
	\end{align*}
	Where $\mathcal N := \mathcal N\left(\Phi^{-1}(\overline{B_{1-\eps}}\setminus B_{1-3\eps/2})\right)$ is an open neighborhood of $\Phi^{-1}\left(\overline{B_{1-\eps}}\setminus B_{1-3\eps/2}\right)$ chosen to be small enough such that $\mathcal N \cap U = \emptyset$ and such that $\Phi(\mathcal N)\subset A^{1}_{1-2\eps}$.
	The set $Q$ is a closed subset of $\Phi^{-1}(A^{1}_{1-2\eps})$ and $W$ is an open neighborhood of $Q$, therefore by the above Lemma \ref{lem:cofib_thicken}, (Applied with $M=\Phi^{-1}(A^{1}_{1-2\eps})$, $U=W$) there exists a closed neighborhood, $Z$, of $Q$, where $Q\subset Z\subset W$ such that the inclusion $Z \hookrightarrow \Phi^{-1}(A^{1}_{1-2\eps})$ is a cofibration.
	Define a homotopy $h_t \colon Z\times [0,1] \to A^{1}_{1-2\eps}$ by:
	\begin{align*}
	h_t(x):=
	\begin{cases*}
	f_t(x) & 	$x\in U$ \\
	\Phi(x) &   $x\in \mathcal N$
	\end{cases*}		
	\end{align*} 
	$h_0 = \Phi\vert^{}_Z \colon Z \to A^{1}_{1-2\eps}$ has an extension $\Phi\vert^{}_{\Phi^{-1}({A^{1}_{1-2\eps}})}\colon \Phi^{-1}(A^{1}_{1-2\eps}) \to A^{1}_{1-2\eps}$ and hence
	by the cofibration property, the homotopy $h_t$ extends to a function $H_t \colon  \Phi^{-1}(A^{1}_{1-2\eps})\times[0,1] \to A^{1}_{1-2\eps}$ such that $H_t\vert^{}_Z = h_t$. By the smooth approximation theorem (Whitney's approximation), since $h_t$ is already smooth we can choose $H_t$ to be smooth.
	Define: \[\widetilde{\Phi} := H_1\colon \Phi^{-1}(A^{1}_{1-2\eps}) \to A^{1}_{1-2\eps}.\]
	$\widetilde{\Phi}$ has the desired properties near $X$ but is not yet defined on all of $M$.
	\\ \\ \textbf{Step 2:} We construct $\widehat{\Phi}\colon M \to \overline{B_1}$, (Now defined on all of $M$) with the following properties:
	\begin{enumerate}
		\item $\widehat{\Phi}\vert^{}_X \equiv \widetilde{\Phi}\vert^{}_X$.
		\item $\widehat{\Phi}^{-1}(B_{1-3\eps}) = \Phi^{-1}(B_{1-3\eps})$.
		\item $\widehat{\Phi}\vert^{}_{\Phi^{-1}(B_{1-3\eps})}\equiv \Phi\vert^{}_{\Phi^{-1}(B_{1-3\eps})}$.
	\end{enumerate}
	We define $\widehat{\Phi}\colon M \to \overline{B_1}$ by
	\begin{align*}
	\widehat{\Phi}:=
	\begin{cases*}
	\Phi(x) & 	$x\in \Phi^{-1}\left(B_{1-5/4\eps}\right)$ \\
	\widetilde{\Phi}(x) &   Otherwise.
	\end{cases*}		
	\end{align*} 
	Since $\widetilde{\Phi}\vert^{}_{\mathcal N} = \Phi\vert^{}_\mathcal{N}$, and since the boundary of $\Phi^{-1}\left(B_{1-5/4\eps}\right)$ is contained in $\mathcal {N}$, the map $\widehat{\Phi}$ is indeed smooth.
	\\ \textbf{Step 3:} We construct $\Psi\colon M \to \overline{B_1}$, using Corollary \ref{cor:pb_bound_retract} to obtain a function with the desired bounds on the Poisson bracket.\\	
	Let $T\colon \mathbb R^2 \to \overline{B_{1-3\eps}}$ be an $\eps$-pseudoretract, and denote by $\mathcal H_{\frac{1}{1-3\eps}}$ the homothety by a factor of $\frac {1}{\sqrt{1-3\epsilon}}$. 
	Define:
	\[
	\Psi:= \mathcal H_{\frac{1}{1-3\eps}} \circ T\circ \widehat{\Phi}.
	\]
	Now, by Proposition \ref{prop:pseudoretract_properties} we have $\left\Vert\left\{\Psi_1, \Psi_2\right\}\right\Vert \le \frac{1+\eps}{1-3\eps}\cdot\left\Vert\left\{\Phi_1, \Phi_2\right\}\right\Vert$.

	By shrinking the neighborhoods $U_k$ of $X_k$ so that $U_k\subset Z$ we have $\Psi(U_k) = f(U_k) \subset \gamma^{}_k$ for all $1\le k \le 3$. Thus $\Psi$ is $\Pb_3(X_1,X_2,X_3)$-admissible. WLOG, we can assume that $\Phi$ is (CS) with respect to $p=(0,0)$ (otherwise we move $p$ by the methods of Proposition \ref{prop:pb_invariance}). Therefore it holds that $\widehat{\Phi}$ is (CS), and hence also $\Psi$, as they are obtained by post-compositions. We have shown that for all $\eps>0$ small enough that:
	\[
	\frac{1+\eps}{1-3\eps}\cdot \Pb_X(\alpha) \ge \Pb_3(X_1,X_2,X_3).
	\]
	By sending $\eps \to 0$ the result follows.
	
	\begin{figure}[h]
		\centering
		\includegraphics[scale=0.37]{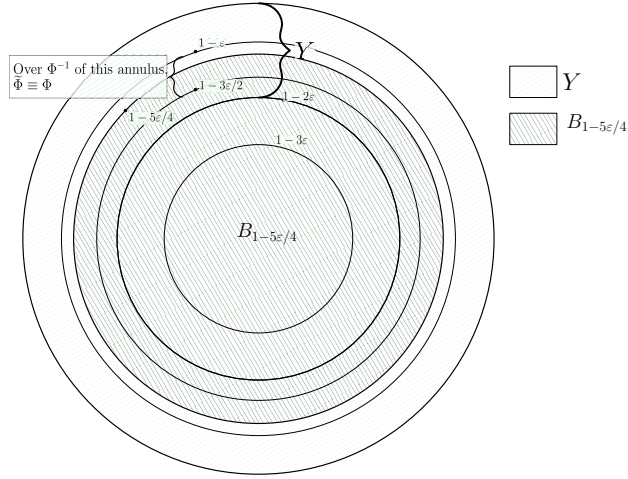}
		\caption{$\overline{B_1}$ and its subsets involved in the proof.}
		\label{fig3}
	\end{figure}
	\end{proof}
	\subsection{Proof of Lemma \ref{lem:cofib_thicken} (Thickening a set to a cofibration)}
	
	In this section we prove the following Lemma:
	\begin{lem*}
			Let $M$ be a manifold, and let $X\subset U\subset M$ such that $X$ is closed and $U$ is open in $M$. Then there exists a closed neighborhood, $Z$, of $X$, such that $X\subset Z \subset U$ and such that the inclusion map $Z\hookrightarrow M$ is a cofibration.
	\end{lem*}
	The following appears in \cite{May} as a corollary of Theorem "HELP" (Homotopy Extension and Lifting Property).
	\begin{prop}
		Let $X$ be a CW-complex and let $i\colon A \to X$ be the inclusion of a subcomplex, then $i$ is a cofibration.
	\end{prop}
	We will also need a theorem of Whitehead about existence of triangulations for smooth manifolds with boundary, see \cite{whitehead1940c1} and \cite{munkres2016elementary}:
	\begin{prop}[Whitehead]
		If $M$ is a smooth para-compact manifold with boundary, then every smooth triangulation of $\partial M$ can be extended to a smooth triangulation of $M$
	\end{prop}
	As a corollary we have:
	\begin{cor}\label{prop:nbhdMBdryCofibration}
		Let $M$ be a $d$-dimensional manifold and let $N$ be a $d$-dimensional connected manifold with boundary. Then any embedding $i\colon N \to M$ is a cofibration.
	\end{cor}
	\begin{proof}
		For brevity identify $N$ with its image in $M$. Consider $N$ and $M\setminus \operatorname{Int} N$, they are both manifolds with a common boundary $\partial N$. Choose a triangulation of $\partial N$. Extending to both $N$ and $M\setminus \operatorname{Int} N$ yields a triangulation of $M$ such that $N$ is a subcomplex. Any triangulation induces CW-structure in the obvious way, therefore $i\colon N \to M$ is a cofibration.
	\end{proof}
	We can now prove Lemma \ref{lem:cofib_thicken}.
	\begin{proof}
		In light of Proposition \ref{prop:nbhdMBdryCofibration} it is enough to show that there exists a closed neighborhood, $Z$, of $X$, contained in $U$, such that $Z$ is an embedded $\dim\!M$-dimensional manifold with boundary.
		Since $X$ is a compact subset of $M$, there exists a smooth function, $h\colon M \to [0,\infty)$, such that $h^{-1}(0) = X$. By Sard's theorem, the critical values of $h$ are of measure $0$ in $[0,\infty)$, therefore there exists a regular value $r\in[0,\infty)$ such that $h^{-1}\left([0,r)]\right) \subset U$. Set $Z = h^{-1}\left([0,r)]\right)$. $Z$ is an embedded $\dim\!M$-dimensional manifold with boundary, therefore $i\colon Z\to M$ is a cofibration.
	\end{proof}

	\subsection{Proof of Theorem \ref{thm:pbaSubhomogeneous} - Subhomogeneity of $\Pb_X(\cdot)$}
	In this section we prove the following:
	\begin{thm}
		Let $X$ be compact subset of a symplectic manifold $M$. Then for all $\alpha \in H^1(X;\mathbb Z)$ and for all $0<k\in \mathbb N$ we have:
		\[
		\Pb^{}_X(k\alpha) \le k \cdot \Pb^{}_X(\alpha).
		\]
	\end{thm}
	\begin{proof}
		Let $\Phi \colon M \to \overline{B_1}$ be a function admissible for $\Pb^{}_X(\alpha)$. We construct a function admissible for $\Pb^{}_X(k\alpha)$ in the following way:
		Consider $R_k\colon \overline{B_1} \to \overline{B_1}$ defined in polar coordinates by
		\[R_k\left(re^{i\theta}\right)=re^{ik\theta}.\]
		This is a smooth function except for the origin.
		For every $\eps > 0$ consider $T\colon \overline{B_1} \to \overline{B_{1-\eps}}$ defined by collapsing the disc $\overline{B_\eps}$ around the origin to point, smoothly, in a similar fashion to what is done in the construction of $\eps$-pseudoretracts, that is
		\[ 
			T(re^{i\theta}) = \rho\left(r\right)e^{i\theta},
		\]
		where $\rho\colon [0,\infty) \to [0,\infty)$ is a function such that $\rho(x)\vert^{}_{[1/2,\infty)} = x - \eps$, $\rho\vert^{}_{[0,\eps]} = 0$ and $0\le\rho^\prime \le 1 + \eps$.
		The function $\rho$ is constant near $0$, therefore the composition $T\circ R_k$ is smooth. Define:
		\[\Psi := \mathcal{H}_{\frac{1}{1-\eps}} \circ T \circ R_k.\]
		$\Psi$ is a function admissible for $\Pb^{}_X(k\alpha)$. Note that
		\[\left\Vert\left\lbrace{\Psi_1,\Psi_2}\right\rbrace \right\Vert\le k\cdot\frac{1+\eps}{1-\eps}\left\Vert\left\lbrace{\Phi_1,\Phi_2}\right\rbrace\right\Vert.\]
		Hence for all $\eps>0$ we have $\Pb^{}_X(k\alpha) \le k \cdot\frac{1+\eps}{1-\eps} \cdot \Pb^{}_X(\alpha)$. The result follows by sending $\eps\to 0$.
	\end{proof}
	
	\bibliographystyle{alpha}

	\bibliography{Bibliography}
	\bigskip
	\medskip

	\noindent Yaniv Ganor \\
	School of Mathematical Sciences \\
	Tel Aviv University \\
	Tel Aviv 69978, Israel \\
	\texttt{yanivgan@post.tau.ac.il} \\
	\medskip
	
\end{document}